\documentclass[12pt, reqno]{amsart}

\usepackage{amsmath,amssymb,amscd,enumerate,eucal,bm}
\usepackage{color}

\newcommand{\R}{\mathbb{R}}

\newcommand{\N}{\mathbb{N}}

\newcommand{\G}{\mathbb{G}}

\newcommand{\cS}{\mathcal{S}}

\newtheorem{theorem}{Theorem}

\newtheorem{corollary}[theorem]{Corollary}
\newtheorem{lemma}[theorem]{Lemma}
\newtheorem{proposition}[theorem]{Proposition}

\theoremstyle{definition}
\newtheorem{definition}[theorem]{Definition}

\newtheorem{remark}[theorem]{Remark}

\numberwithin{theorem}{section}
\numberwithin{equation}{section}

\newcommand{\dd}{\; \mathrm{d}}

\begin{document}

\title[Function Spaces via Fractional Poisson Kernel]{Function Spaces via Fractional Poisson Kernel on Carnot Groups and Applications}

\date{\today}

\author[Ali Maalaoui]{Ali Maalaoui}
\address[Ali Maalaoui]{Department of Mathematics and Natural Sciences, American University of Ras Al Khaimah, PO Box 10021, Ras Al Khaimah, UAE}
\email[Ali Maalaoui]{ali.maalaoui@aurak.ac.ae}

\author[Andrea Pinamonti]{Andrea Pinamonti}
\address[Andrea Pinamonti]{Department of Mathematics, University of Trento, Via Sommarive 14, 38123 Povo (Trento), Italy}
\email[Andrea Pinamonti]{Andrea.Pinamonti@unitn.it}

\author[Gareth Speight]{Gareth Speight}
\address[Gareth Speight]{Department of Mathematical Sciences, University of Cincinnati, 2815 Commons Way, Cincinnati, OH 45221, United States}
\email[Gareth Speight]{Gareth.Speight@uc.edu}

\begin{abstract}
We provide a new characterization of homogeneous Besov and Sobolev spaces in Carnot groups using the fractional heat kernel and Poisson kernel. We apply our results to study commutators involving fractional powers of the sub-Laplacian.
\end{abstract}

\maketitle

\section{Introduction}

Besov and Sobolev spaces measure regularity of functions and are of central importance in the study of PDEs. There has been much work on these spaces and their characterization in different settings. Such alternative characterizations provide flexibility for applications. In this work, we give a new characterization of homogeneous Besov and Sobolev spaces in Carnot groups using the fractional heat kernel and Poisson kernel. We use this to study commutators involving fractional powers of the sub-Laplacian.


A Carnot group is a Lie group whose Lie algebra admits a stratification. This decomposes the Lie algebra as a direct sum of vector subspaces, the first of which is called the horizontal layer and generates the other subspaces via Lie brackets. Carnot groups have a rich geometric structure adapted to the horizontal layer, including translations, dilations, Carnot-Carath\'{e}odory (CC) distance, and a Haar measure \cite{ABB, CDPT07, Gro96, Mon02}. Carnot groups have been studied in contexts such as differential geometry \cite{CDPT07}, subelliptic differential equations \cite{Bon, Fol1, Fol2}, real and complex analysis \cite{S,PS16,PSZ19}. For an introduction to Carnot groups from the point of view of this paper and for further examples, we refer to \cite{Bon, Fol1, S}.

In the Euclidean case there have been characterizations of Besov and Sobolev spaces using multiple tools, mostly relying on the Fourier transform and Littlewood-Paley decompositions \cite{candy, Bu1,Bu2}. In Carnot groups there have been a few characterizations of such spaces, for instance using the heat kernel \cite{Saka} and a spectral multiplier version of Besov spaces \cite{GY}. We also mention the use of a Littlewood-Paley decomposition in the study of the phase space in the Heisenberg group \cite{Ba1,Ba}. This uses the Fourier transform in that setting. We also point out the extension of the characterization in \cite{Saka} to the case of metric measure spaces with heat kernels satisfying a Gaussian bound, \cite{CG}. As we will see later on, the kernels that we will be using do not satisfy this bound.

The heat kernel and fractional heat kernel in Carnot groups have been studied for some time, e.g. see \cite{Fer} and the references therein. The Poisson kernel in Carnot groups was introduced and studied in \cite{FS}, but the fractional one is a recent discovery. It was first introduced and studied in \cite{Fran} to exhibit a Harnack type estimate for the fractional Laplacian. The method of construction follows the classical one introduced by Caffarelli and Silvestre in \cite{CS}, but here using the spectral resolution of the sub-Laplacian. We also point out that there is also another construction for a different fractional Poisson kernel in the Heisenberg group for the conformal fractional sub-Laplacian in \cite{FGMT} and the construction relies mainly on the Fourier transform.

In this paper we start by defining a norm using the fractional heat kernel which ends up being equivalent to the classical homogeneous Besov norm as stated in Theorem \ref{pp}. This procedure is close to the one of \cite{Saka} and it does rely partially on the semi-group property of the fractional heat kernel. Next, we study different properties of the fractional Poisson kernel, allowing us, as stated in Theorem \ref{Oct31}, to provide different equivalent norms to the classical homogeneous Besov spaces. The main challenge in this procedure is to bypass the use of the Fourier transform and still keep certain harmonic analysis properties of the different kernel we are considering. Also, in the same spirit, in Theorem \ref{Oct23}, we provide a lower bound for the fractional Sobolev norms using a square function type quantity involving the convolution with the Poisson kernel and we finish in Proposition \ref{bmo}, by providing a characterization of the BMO norm.

Concerning applications of our results, the second characterization that we provide for Besov and Sobolev spaces appears to be well suited to the study of commutators involving fractional powers of the sub-Laplacian. We recall that, in \cite{M}, the first author provided a family of estimates for the commutator of the fractional sub-Laplacian using a more direct approach in estimating the singular kernel of the operator. In this work we provide an extended result, which generalizes many classical commutator estimates known in the Euclidean setting to the case of Carnot groups. For instance, in Theorems \ref{lp} and \ref{riv} we provide bilinear type estimates for three terms commutators involving the fractional sub-Laplacian. In fact the first result (namely Theorem \ref{lp}) provides $L^{p}$ type estimates and the second result deals with the borderline setting of bounding the Hardy norm. Also, in Theorem \ref{chan}, we provide a proof of the Chanillo type commutator estimates for the Carnot group setting.  We follow closely the ideas provided in the Euclidean setting \cite{lenz} to use the fractional Poisson kernel to simplify the expressions of the commutators. But we point out that in the Euclidean case, the estimates and characterizations of the different spaces was established separately in \cite{candy}. This is why, in our case, we first have to cross the difficulty of characterizing these spaces.

In general, commutator estimates are a fundamental tool in the study of the regularity of PDE, especially in the fractional setting. For instance, in Carnot groups, \cite{M} gives applications to the study of the regularity and decay of solutions to the fractional CR-Yamabe problem, while \cite{MMC} characterizes the asymptotic profile decomposition of Palais-Smale sequences for the same problem. In the Euclidean setting one has even more applications of commutator estimates \cite{Ken,Schik, DR1,DR2}. 

The structure of the paper is as follows. 

In Section \ref{preliminaries} we provide the necessary preliminaries on the structure of Carnot groups, the sub-Laplacian and the heat kernel. 

In Section \ref{fractionalheat} we provide a characterization of Besov spaces using the fractional heat kernel. This is the kernel of the flow generated by the fractional power of the sub-Laplacian. The proof in this section follows the approach in \cite{Saka}, where an analogous characterization of Besov spaces was obtained using the standard (non-fractional) heat kernel and Poisson kernel.

In Section \ref{fractionalpoisson} and Section \ref{bounds} we move to the characterization of Besov, Sobolev and BMO spaces using the fractional Poisson kernel. Here we generalize ideas in the Euclidean setting and avoid notions involving the Fourier transform because that is a tool that we cannot afford in Carnot groups in general. 

In Section \ref{applications} we provide several applications of our results to estimates for commutators of fractional powers of the sub-Laplacian. Such estimates were established and studied in the Euclidean setting in \cite{Chan,DR1,DR2,lenz,Ken,Schik}.

\vspace{.5cm}

\textbf{Acknowledgements:} A. Pinamonti is a member of {\em Gruppo Nazionale per l'Analisi Ma\-te\-ma\-ti\-ca, la Probabilit\`a e le loro Applicazioni} (GNAMPA) of the {\em Istituto Nazionale di Alta Matematica} (INdAM). 

This work was supported by a grant from the Simons Foundation (\#576219, G. Speight).

\section{Preliminaries}\label{preliminaries}

\subsection{Carnot Groups}

\begin{definition}
A connected and simply connected Lie group $(\G,\cdot)$ is a {\it Carnot group  of
step $k$} if its Lie algebra ${\mathfrak{g}}$  admits a {\it step $k$ stratification}. This means that
there exist non-trivial linear subspaces $V_1,\dots ,V_k$ of $\mathfrak{g}$ such that
\begin{equation}\label{stratificazione}
{\mathfrak{g}}=V_1\oplus \dots \oplus V_k
\end{equation}
where $[V_1,V_i]=V_{i+1}$ for $1\leq i<k$ and $[V_{1},V_{k}]=\{0\}$. Here $[V_1,V_i]$ is the subspace of ${\mathfrak{g}}$ generated by
the commutators $[X,Y]$ with $X\in V_1$ and $Y\in V_i$. 
\end{definition}

Let $m_i=\dim(V_i)$ for $i=1,\dots,k$. Define $h_0=0$ and $h_i=m_1+\dots +m_i$ for $i=1,\ldots, k$. We also use the notation $n:=h_k$ and $m:=m_1$.
The {\em homogeneous dimension} of $\G$ is then defined by $Q:=\sum_{i=1}^{k} i \dim(V_i)$. 

Choose a family of left invariant vector fields $X=\{X_1,\dots,X_{n}\}$ adapted to the stratification of $\mathfrak g$, i.e. such that $X_{h_{j-1}+1},\dots, X_{h_j}$ is a basis of $V_j$  for each $j=1,\dots, k$. This identifies $\mathfrak{g}$ with $\mathbb{R}^{n}$. Using exponential coordinates of the first kind we identify $\G$ with $\mathfrak{g}$ and hence with $\mathbb{R}^{n}$. With these coordinates, $X_i(0)=e_i$ for $i=1,\dots,n$. 

\begin{definition}
The sub-bundle of the tangent bundle $T\G$ that is spanned by the
vector fields $X_1,\dots,X_{m}$ plays a particularly important
role in the theory. It is called the {\em horizontal bundle}
$H\G$. The fibers of $H\G$ are 
\[ 
H_x\G=\mbox{span}\{X_1(x),\dots,X_{m}(x)\},\qquad x\in\G .
\] 
We can endow each fiber of $H\G$ with a corresponding inner 
product $\left\langle \cdot, \cdot\right\rangle$ and with a norm
$\vert\cdot\vert$ that make the basis $X_1(x),\dots,X_{m}(x)$
an orthonormal basis. 
The sections of $H\G$ are called {\em horizontal
sections} and a vector of $H_x\G$ a {\em horizontal vector}. 
Each horizontal section is identified by its canonical coordinates
with respect to this moving frame $X_1(x),\dots,X_{m}(x)$. This
way, a horizontal section $\phi$ is identified with a function
$\phi =(\phi_1,\dots,\phi_{m}) :\R^n \rightarrow \R^{m}$.
\end{definition}

\begin{definition}
For any $x\in\G$, the {\em left translation} $\tau_x:\G\to\G$ is defined by $\tau_x z=xz$.

For any $\lambda >0$, the {\em dilation} $\delta_\lambda:\G\to\G$, is defined as
\begin{equation}\label{dilatazioni}
\delta_\lambda(x)=
(\lambda \xi_1,\dots ,\lambda^k\xi_k),
\end{equation} 
where $x=(\xi_1,\dots,\xi_k)\in\R^{m_1}\times\dots \times \R^{m_k}\equiv\G$.

The \emph{Haar measure} of $\G=(\R^n,\cdot)$ is the Lebesgue measure
in $\R^n$. If $A\subset \G$ is Lebesgue measurable, we
write $|A|$ to denote its Lebesgue measure.
\end{definition}

Let $|\cdot|:\G\to [0,\infty)$ denote a \emph{symmetric homogeneous norm} on $\G$ \cite{Bon}, meaning:
\begin{itemize}
\item $|\cdot|$ is continuous,
\item $|\delta_{\lambda}(x)|=\lambda |x|$ for every $\lambda>0$,
\item $|x^{-1}|=|x|$.
\end{itemize}

Note that any two continuous homogeneous norms are equivalent, i.e. within constant multiplicative constant factors of each other. All the estimates we give are the same if the norm is changed up to changes in constants. We denote the ball centered at a point $x\in \G$ with radius $r>0$ by
\[
B(x,r)=\{y\in \G: | y^{-1} x| < r\}.
\]
We denote balls centered at the identity $0$ by $B(r)=B(0,r)$. 

We shall write $f(x)\lesssim g(x)$ if there exists a constant $C$ such that $f(x)\leq C g(x)$ for all $x\in \G$, or equivalently, $\sup_{x\in G}\frac{f(x)}{g(x)}\leq C <\infty$.
Similarly we shall write $f(x)\approx g(x)$ if $f(x) \lesssim g(x)$ and $g(x)\lesssim f(x)$.

\begin{definition}
Suppose $f:\G\to \R$ is a function for which $X_{j} f$ exists for $1\leq j\leq m$. Then we define the \emph{horizontal gradient} of $f$ as the horizontal section whose coordinates are $(X_1f, \dots ,X_{m}f)$:
\begin{equation*}
\nabla_{\G}f:=\sum_{i=1}^{m}(X_if)X_i.
\end{equation*}
We denote by $\Delta_b$ the \emph{positive sub-Laplacian} defined by
\[
\Delta_{b}f:= \sum_{j=1}^m X_jX_j f
\]
whenever $f$ is a function such that $X_jX_j f$ exists for $1\leq j\leq m$.
\end{definition}

If $\Omega \subset \G$ is an open set, we define $C^{\infty}(\Omega)$ as in the classical case when $\Omega$ is a subset of $\R^n$. We will use the inequality
\[\|f(\cdot y)-f(\cdot)\|_{L^{1}}\lesssim |y|\|\nabla_{\G}f\|_{L^{1}}\]
for all $y\in \G$ and sufficiently smooth $f\colon \G \to \R$. This is a consequence of the Fundamental Theorem of Calculus.


\subsection{Heat Kernel}

For every multi-index $\beta=(\beta_1,\ldots,\beta_n)\in \mathbb{N}^n$, we denote
\[|\beta|=\beta_1+\cdots+\beta_n\]
and
\[D^\beta=X_{1}^{\beta_{1}}\cdots X_{n}^{\beta_{n}}\]
where
\[
X_i^{\beta_{i}}=\underbrace{X_iX_i\cdots X_i}_{\beta_i-times}.
\]
We also use the notation $(\partial/\partial x_{i})^{\beta}$ or $\partial^{\beta}$ to denote differentiation with respect to the standard basis of $\R^n$. The Schwartz space and space of distributions are defined as in the classical setting, which we now briefly recall.

\begin{definition}
We define the \emph{Schwartz space} $\mathcal{S}(\mathbb{G})$ by identification of $\G$ with $\R^{n}$:
\begin{align*}
\cS(\G)=& \{ \phi \in C^{\infty}(\G)\colon P(\partial/\partial x_{i})^{\beta}\phi \quad \mbox{is bounded on }\G \\
& \qquad \mbox{ for every polynomial }P \mbox{ and every multi-index }\beta\}.
\end{align*}
We equip $\cS(\G)$ with the following seminorms for multi-indices $\alpha, \beta\in\mathbb{N}^n$:
\[\|\phi\|_{\alpha,\beta}=\sup_{x\in \G} |x^{\alpha}D^{\beta}\phi|.\]
\end{definition}

The convolution of two functions $f,g:\G\to \R$ is defined whenever it makes sense by
\[ (f\ast g)(x)=\int_{\R^n} f(xy^{-1})g(y) \dd y = \int_{\R^n}g(y^{-1} x)f(y)\dd y.\]

\begin{definition}
The continuous dual of $\cS(\G)$ with the family of seminorms $\|\cdot \|_{\alpha, \beta}$ is the space of \emph{distributions} on $\G$, denoted $\cS'(\G)$. 

The action of a distribution $f$ on a Schwartz function $\phi$ is denoted $\langle f, \phi \rangle$. The \emph{convolution} $f\ast \phi$ of $f$ and $\phi$ is defined by $(f\ast \phi)(x)=\langle f, \tilde{\phi}\rangle$ where $\tilde{\phi}(y)=\phi(y^{-1}x)$. If $\alpha$ is a multi-index, the \emph{derivative} $\partial^{\alpha}f$ of a distribution $f$ is defined by
\[\langle \partial^{\alpha}f, \varphi\rangle = (-1)^{|\alpha|}\langle f, \partial^{\alpha}\varphi\rangle \quad \varphi \in \cS(\G).\]
\end{definition}

Define the \emph{parabolic version} of a Carnot group $\G$ by $\hat\G:=\R\times\G$. This is a Carnot group where the group operation in the first coordinate 
is the usual addition and its homogeneous dimension is $Q+2$. We define dilations on $\hat\G$ by $\hat{\delta}_\lambda(t,x)=(\lambda^2 t,\delta_\lambda(x))$. 

\begin{definition}
The \emph{heat operator} is the operator $\mathcal{H}$ on $\hat{\G}$ defined by $\mathcal{H}:=\partial_t + \Delta_b$.
\end{definition}

The heat operator is:
\begin{itemize}
\item translation invariant i.e. for any $g\in\G$, $\mathcal{H}(u \circ \tau_g)=(\mathcal{H}(u))\circ \tau_g$,
\item homogeneous of degree $2$ i.e. for any $\lambda>0$, $\mathcal{H}(u\circ\hat{\delta}_\lambda)= \lambda^2 \mathcal{H}u$,
\item hypoelliptic i.e. if $u$ is a distribution on $\hat {\G}$ such that $\mathcal{H}u$ is $C^{\infty}$ in some open set $\Omega$, then $u$ must be $C^{\infty}$ on $\Omega$.
\end{itemize}

\begin{definition}
The heat operator $\mathcal{H}$ admits a fundamental solution $h$, usually called the \emph{heat kernel} for $\G$.

Write $h_{t}(x):=h(t,x)$ and define for $f$ locally integrable on $\mathbb{R}^{n}$:
\[H_{t}f(x)=(f\ast h_{t})(x)=\int_{\mathbb{R}^{n}} h(t,y^{-1}x)f(y)\dd y\]
whenever the integral exists. Then $\{H_{t}\}_{t>0}$ is called the \emph{heat semigroup} for $\G$.
\end{definition}

We now recall properties of $h$ and $H_{t}$, see for example \cite{Bon, FS} or \cite[Section IV. 4]{VSCC}

\begin{theorem}\label{heatProp}
The heat kernel $h$ satisfies:
\begin{enumerate}
	\item $h\in C^{\infty}(\hat{\G}\setminus \{(0,0)\})$;
	\item $h(\lambda^2 t, \delta_{\lambda} (x))=\lambda^{-Q} h(t,x)$ for every $x\in \G$ and $t,\lambda>0$;
	\item $h(t,x)=0$ for every $t<0$ and $\int_{\G} h(t,x) \dd x=1$ for every $t>0$;
	\item $h(t,x)=h(t,x^{-1})$ for every $t>0$ and $x\in \G$;
	\item there exists $c\geq 1$ (depending only on $\mathbb{G}$) such that 
	for every $x\in\G$ and $t>0$ 
\begin{equation}\label{heat_estimates}
c^{-1}t^{-Q/2}\exp \Big(-\frac{ |x|^2)}{c^{-1}t}\Big)\leq h(t,x)\leq c t^{-Q/2}\exp \Big(-\frac{ |x|^2}{ct}\Big);
\end{equation}
\item For every nonnegative integer $k$ and $\beta\in \mathbb{N}^n$, there exists $c=c(\beta,k)>0$ such that for every $x\in\mathbb{G}$ and $t>0$
\begin{equation}\label{heat_estimates_der}
\left|\frac{\partial^k}{\partial t^k} D^{\beta} h(t,x)\right|\leq c t^{-\frac{Q+j+2k}{2}}e^{-\frac{|x|^2}{t}}
\end{equation}
\end{enumerate}
Further, for any $f\in L^{1}(\mathbb{R}^{n})$ and $t>0$, we have $H_{t}f \in C^{\infty}(\mathbb{R}^{n})$ and $u(t,x)=H_{t}f(x)$ solves $\mathcal{H}u=0$ in $(0,\infty) \times \mathbb{R}^{n}$. Also $u(t,x)\to f(x)$ strongly in $L^{1}(\mathbb{R}^{n})$ as $t\to 0$.
\end{theorem}

We now recall that the fractional sub-Laplacian and its inverse can be expressed using the heat semigroup $H_t$ as in \cite{Fol1}.

\begin{definition}
We define the \emph{fractional sub-Laplacian} by
\begin{align}\label{defLap}
(-\Delta_{b})^{\alpha}f=\lim_{\varepsilon \to 0}\frac{1}{\Gamma(1-\alpha)}\int_{\varepsilon}^{\infty}t^{-\alpha}(-\Delta_{b})H_{t}f  \dd t
\end{align}
and
$$(-\Delta_{b})^{-\alpha}f=\lim_{\eta\to\infty}\frac{1}{\Gamma(\alpha)}\int_{0}^{\eta}t^{\alpha-1}H_{t}f \dd t,$$
where $0<\alpha<1$ and $f\in L^2(\G)$ is any function for which the relevant limit exists in $L^2$ norm. 
\end{definition}

We recall the following proposition \cite{Fol1,FS}.

\begin{proposition}\label{above}
For $0<\alpha<Q$ the integral
$$R_{\alpha}(x)=\frac{1}{\Gamma(\frac{\alpha}{2})}\int_{0}^{\infty}t^{\frac{\alpha}{2}-1}h(t,x)\dd t, \qquad x\in \G,$$
converges absolutely and has the following properties:
\begin{itemize}
\item $R_{\alpha}$ is a kernel of type $\alpha$, i.e. is $C^{\infty}$ away from $0$ and homogeneous of degree $\alpha - Q$;
\item $R_{\alpha}*R_{\beta}=R_{\alpha+\beta}$ for $\alpha, \beta>0$ and $\alpha+\beta<Q$;
\item $R_{2}$ is the fundamental solution of $-\Delta_{b}$, i.e. $(-\Delta_{b})R_{2}=\delta_{0}$;
\item For $f\in L^{p}(\mathbb{G})$ and $1<p<\infty$, we have $(-\Delta_{b})^{-\alpha}f=f*R_{2\alpha}$.
\end{itemize}
\end{proposition}

From Proposition \ref{above} and Theorem \ref{heatProp} it follows that $R_{\alpha}(x)\approx |x|^{-Q+\alpha}$. Also the function $\rho(x)=(R_{\alpha}(x))^{\frac{1}{\alpha-Q}}$ defines a symmetric homogeneous norm which is smooth away from the origin and induces a quasi-distance equivalent to the left-invariant Carnot-Carath\'{e}odory distance.

In a similar way one can define the function $\tilde{R}_{\alpha}$, introduced in \cite{Fran}, for $\alpha<0$ and $\alpha \not \in \{0,-2,-4,\cdots\}$ by
$$\tilde{R}_{\alpha}(x)=\frac{\frac{\alpha}{2}}{\Gamma(\frac{\alpha}{2})}\int_{0}^{\infty}t^{\frac{\alpha}{2}-1}h(t,x)\dd t.$$
Again, $\tilde{R}_{\alpha}$ is homogeneous of degree $\alpha-Q$ and 
\begin{equation}\label{sx}
\tilde{R}_{\alpha}(x)\approx |x|^{\alpha-Q}.
\end{equation}
Using classical interpolation (or what is it called $\lambda-$kernel estimates in \cite{FS}) one has for $0<\alpha<Q$,
\begin{equation}\label{stimaR}
\|R_{\alpha} u\|_{L^p}\lesssim \|u\|_{L^q}
\end{equation}
for $\frac{1}{p}=\frac{1}{q}-\frac{\alpha}{Q}$ and $1<q<Q$.
Using $\tilde{R}_{\alpha}$ one can define another representation for the fractional sub-Laplacian. The following theorem is from \cite[Theorem 3.11]{Fran}.


\begin{theorem}\label{raplap}
If $u\in \mathcal{S}(\G)$, then for $0<\alpha<1:$
\begin{align*}
(-\Delta_{b})^{\alpha}u(x)&=P.V.\int_{\mathbb{G}}(u(y)-u(x))\tilde{R}_{-2\alpha}(y^{-1}x)\dd y\\
&=\lim_{\varepsilon\to 0^{+}}\int_{\mathbb{G}\setminus B(x,\varepsilon)}(u(y)-u(x))\tilde{R}_{-2\alpha}(y^{-1}x)\dd y.
\end{align*}
\end{theorem}
Moreover, using Balakrishnan's approach (see \cite{BalT,Bal}) and what proved in \cite[Lemma 8.5]{Gar1} (see also \cite{Fer}) we also have the following formula 
\begin{theorem}
If $u\in \mathcal{S}(\G)$, then for $0<\alpha<1:$
\begin{align*}
(-\Delta_{b})^{\alpha}u(x)&=-\frac{\alpha}{\Gamma(1-\alpha)}\int_0^\infty t^{-\alpha-1}(H_t u(x)-u(x))\, dt\\
&=\frac{1}{\Gamma(-\alpha)}\int_0^\infty t^{-\alpha-1}(H_t u(x)-u(x))\, dt.
\end{align*}
The integral in the right-hand side must be interpreted as a Bochner integral in $L^2(\G)$.
\end{theorem}
\subsection{Spectral Analysis in Carnot Groups}
We collect here some well-known results in spectral analysis which will be used later in the paper.

Since $-\Delta_b$ is self-adjoint with domain $\{f\in L^2(\mathbb{G}) : -\Delta_b f\in L^2(\mathbb{G})\}$,
we can consider its spectral resolution $\int_0^{\infty} \lambda dE(\lambda)$. Then \cite[(3.12)]{Fol1},
\[
(-\Delta_{b})^{\alpha}=\int_0^{\infty} \lambda^{\alpha} \dd E(\lambda),
\]
with domain
\[
W^{2\alpha,2}(\mathbb{G})=\left\{u\in L^2(\mathbb{G}): \int_0^{\infty} \lambda^{2\alpha} \dd\left\langle E(\lambda)u,u\right\rangle<\infty\right\}.
\]
Any bounded Borel measurable function $m$ on $[0,\infty)$ defines an operator on $L^2(\mathbb{G})$ by
\[
m(-\Delta_b)=\int_0^{\infty} m(\lambda) \dd E(\lambda).
\]
Let $K_m$ denote the \emph{convolution kernel} of the operator $m(-\Delta_b)$, namely $K_{m}$ is a distribution on $\G$ satisfying
\begin{equation}\label{convkernel}
m(-\Delta_b)u=u*K_m \quad \text{for }u\in \mathcal{S}(\mathbb{G}).
\end{equation}
If $m$ is also compactly supported, then $K_m\in L^2(\mathbb{G})$ and there exists a regular Borel measure $\sigma_m$ on $[0,\infty)$, whose support is the $L^2$ spectrum of $-\Delta_b$, such that \cite[Theorem 3.10]{Martini}:
\[
\int_{\mathbb{G}}|K_m(x)|^2 \dd x=\int_0^{\infty} |m(\lambda)|^2 \dd \sigma_m(\lambda).
\]

\begin{remark}
For any function $f$ on $\mathbb{G}$ and any $\lambda >0$, set $d_{\lambda}f(x)=f(\delta_{\lambda}(x))$. We claim:
\begin{equation}
d_{\lambda}^{-1}(-\Delta_b)^{\alpha} d_{\lambda}= \lambda^{2\alpha}(-\Delta_b)^{\alpha}.
\end{equation}
Indeed, by Theorem \ref{raplap} we get
\begin{align*}
d_{\lambda}^{-1}(-\Delta_b)^{\alpha} d_{\lambda}f(x)&=\lim_{\varepsilon\to 0^{+}}\int_{\mathbb{G}\setminus B(\delta_{\frac{1}{\lambda}}x,\varepsilon)}(d_{\lambda} f(y)-d_{\lambda}f(\delta_{\frac{1}{\lambda}}x))\tilde{R}_{-2\alpha}(y^{-1}\delta_{\frac{1}{\lambda}}x)\dd y\\
&=\lim_{\varepsilon\to 0^{+}}\int_{\mathbb{G}\setminus B(\delta_{\frac{1}{\lambda}}x,\varepsilon)}(f(\delta_{\lambda}y)-f(x))\tilde{R}_{-2\alpha}(\delta_{\frac{1}{\lambda}}((\delta_{\lambda}y)^{-1}x))\dd y\\
&=\lambda^{Q+2\alpha}\lim_{\varepsilon\to 0^{+}}\int_{\mathbb{G}\setminus B(\delta_{\frac{1}{\lambda}}x,\varepsilon)}(f(\delta_{\lambda}y)-f(x))\tilde{R}_{-2\alpha}((\delta_{\lambda}y)^{-1}x)\dd y
\end{align*}
where in the last equality we used $\tilde{R}_{-2\alpha}(\delta_{\frac{1}{\lambda}}((\delta_{\lambda}y)^{-1}x)=\lambda^{Q+2\alpha}\tilde{R}_{-2\alpha}((\delta_{\lambda}y)^{-1}x)$. The conclusion follows by a change of variables. 
\end{remark}
Similar to above, one can check that for any $m\in L^{\infty}((0,\infty))$ we have
\begin{equation}\label{identityC}
d_{\lambda}^{-1}m((-\Delta_b)^{\alpha}) d_{\lambda}= m(\lambda^{2\alpha}(-\Delta_b)^{\alpha}).
\end{equation}


\begin{remark}
Given $m\in L^{\infty}((0,\infty))$ and $t>0$, set
\[ \hat{m}(\lambda)=m(\sqrt{\lambda}) \quad \mbox{and} \quad m^t(\lambda)=m(t\sqrt{\lambda}).\] For any $f\in \mathcal{S}(\mathbb{G})$,
\begin{align*}
m^t((-\Delta_b))f&=m(t(-\Delta_b)^{\frac{1}{2}})f\\
&=d_{t}^{-1}m((-\Delta_b)^{\frac{1}{2}})d_{t} f\\
&=d_{t}^{-1}\hat{m}((-\Delta_b))d_{t} f\\
&=d_{t}^{-1}(d_t f* K_{\hat{m}})\\
&=t^{-Q}(f*d_t^{-1}K_{\hat{m}}).
\end{align*}
Hence
\begin{align}\label{stimafond}
K_{m^t}(x)=t^{-Q} K_{\hat{m}}(\delta_{\frac{1}{t}}(x)).
\end{align}
\end{remark}

Now suppose that $\alpha>Q/2$ and fix $\eta\in C^{\infty}_0(0,\infty)$ not identically zero. If $m$ satisfies
\begin{align}\label{condC}
\sup_{t}\|\eta(\cdot)m(t\cdot)\|_{W^{\alpha,2}(\mathbb{R})}<\infty,
\end{align} 
then by \cite[Lemma 6]{Cr}, $K_{m}^{t}\in L^1(\mathbb{G})$ uniformly in $t\in (0,\infty)$. Here $W^{\alpha,2}(\mathbb{R})$ denotes the standard fractional Sobolev space of order $\alpha$.

\subsection{Semigroups}
We recall a few properties of the semigroups generated by fractional powers of generators of 
strongly continuous semigroups. We refer to \cite[Section 11, Chapter IX]{Yosida} for more information and all the missing proofs.

\subsubsection{General Semigroups}
It is well known that $(\Delta_{b})^{\alpha}$ is the generator of a Markovian semigroup $\{e^{t{A}_\alpha}\}_{t>0}$ which is related to $\{e^{tA}\}_{t>0}$ by the subordination formula 
\begin{align}\label{fracheatsgrp1}
e^{t{A}_\alpha}u &=\int_{0}^{\infty}f_{t,\alpha}(s)e^{sA}u  \dd s \\
\label{fracheatsgrp2}
&=\int_{0}^{\infty}f_{1,\alpha}(\tau)e^{\tau t^{1/\alpha}A}u \dd \tau 
\end{align}
where 
\[
f_{t,\alpha}(\lambda)=\left\{\begin{array}{ll} 
{\displaystyle 
\frac{1}{2\pi i}\int_{\sigma-i\infty}^{\sigma+i\infty}e^{z\lambda-tz^\alpha}\dd z} \quad &\mbox{if} \quad  \lambda\geq 0,
\mbox{ } \\
\\
0 & \mbox{if} \quad  \lambda<0.
\end{array}\right.
\]
for $\sigma>0, t>0$ and $0<\alpha<1$.
Thanks to \cite[Proposition 2 of Section 11, Chapter IX]{Yosida}, $f_{t,\alpha}(\lambda)$ is nonnegative for $\lambda \geq 0$ and for $\lambda > 0$
\begin{align}
\int_{0}^\infty f_{t,\alpha}(s) e^{-s\lambda}\dd s=e^{-t\lambda^{\alpha}}.
\end{align}
Moreover,  
\begin{align}\label{stimaz}
f_{t,\alpha}(s)\leq \mathrm{min}\left\{\frac{1}{t^{1/\alpha}}, \frac{t}{s^{1+\alpha}}\right\}
\end{align}
and for $-\infty<\delta<\alpha$
\begin{align}
\int_0^{\infty} f_{t,\alpha}(s) s^{\delta}\dd s=\frac{\Gamma(1-\delta/\alpha)}{\Gamma(1-\delta)}t^{\delta/\alpha}
\end{align}
and if $\delta\geq \alpha$
\begin{align}
\int_0^{\infty} f_{t,\alpha}(s) s^{\delta}\dd s=+\infty.
\end{align}

\subsubsection{Heat Semigroup}
Let us now come to the heat semigroup. In this case, $(-\Delta_b)^\alpha$ is defined in \eqref{defLap} and its domain is $W^{2\alpha,2}(\mathbb{G})$. We may use \eqref{fracheatsgrp1} and the equation $e^{t(-\Delta_{b})}u=H_{t}u$ to write
\[
e^{t(-\Delta_b)^\alpha}u(x)=\int_{0}^{\infty}f_{t,\alpha}(s)\left(\int_{\mathbb{G}}h(s,y^{-1}x)u(y)\dd y\right)\dd s \qquad \mbox{for } u\in L^2(\mathbb{G}).
\]
Hence, using Theorem \ref{heatProp} and \eqref{fracheatsgrp2}, we have
\begin{align*}
e^{t(-\Delta_b)^\alpha}u(x)&=\int_{0}^{\infty}e^{\tau t^{1/\alpha}{(-\Delta_b)}}u(x)f_{1,\alpha}(\tau)\dd \tau\\
&=\int_{0}^{\infty}\left(\int_{\mathbb{G}}h(\tau t^{1/\alpha},y^{-1}x)u(y)\dd y\right)f_{1,\alpha}(\tau)\dd \tau
\\
&=\int_{\mathbb{G}}\left(\int_{0}^{\infty}h(\tau t^{1/\alpha},y^{-1}x)f_{1,\alpha}(\tau)\dd \tau\right) u(y)\dd y.
\end{align*}
Thus, the function 
\begin{equation}\label{fracheatsgrp3}
h_\alpha(t,y)=\int_{0}^{\infty}h(\tau t^{\frac{1}{\alpha}},y)f_{1,\alpha}(\tau)\dd \tau
\end{equation}
is the integral kernel of the semigroup $e^{t(-\Delta_b)^\alpha}$, i.e., 
\begin{equation}\label{fracheatsgrp4}
e^{t(-\Delta_b)^\alpha}u(x)=\int_{\mathbb{G}}h_\alpha(t,y^{-1}x)u(y) \dd y \qquad \mbox{for }u\in L^2(\mathbb{G}).
\end{equation}

\subsection{Besov Spaces}

We now recall the definition of the Besov space $B_{p,q}^{s}(\G)$ \cite{Saka}.

\begin{definition}
Let $0<s<1$, $1\leq p \leq \infty$ and $1\leq q \leq\infty$. 

The Besov space $B_{p,q}^{s}(\G)$ is defined for $q<\infty$ by
$$B_{p,q}^{s}(\G):=\left\{f\in L^{p}(\G): \int_{\G}\Big(\frac{\|f(xy)-f(x)\|_{L^{p}}}{|y|^{s}}\Big)^{q}\frac{\dd y}{|y|^{Q}}<\infty\right\}.$$

The Besov space $B_{p,\infty}^{s}(\G)$ is defined by
$$B_{p,\infty}^{s}(\G):=\left\{f\in L^{p}(\G): \sup_{y\neq 0}\frac{\|f(xy)-f(x)\|_{L^{p}}}{|y|^{s}}<\infty\right\}.$$

We define the corresponding semi-norms by
$$\|f\|_{\dot{B}^{s}_{p,q}}:=\left\{\begin{array}{ll}
\Big(\int_{\G}\Big(\frac{\|f(xy)-f(x)\|_{L^{p}}}{|y|^{s}}\Big)^{q}\frac{\dd y}{|y|^{Q}}\Big)^{\frac{1}{q}} &\qquad \text{if } q<\infty,\\
\sup_{y\neq 0}\frac{\|f(xy)-f(x)\|_{L^{p}}}{|y|^{s}} &\qquad \text{if } q=\infty.
 \end{array}
\right.
$$
\end{definition}

Note $B_{p,q}^{s}$ can also be defined as the completion of $\mathcal{S}(\mathbb{G})$ with respect to $\|\cdot\|_{L^{p}}+\|\cdot\|_{\dot{B}^{p,q}_{s}}$.

\section{Besov Spaces via Fractional Heat Kernel}\label{fractionalheat}

In this section we will provide a characterization of Besov spaces using the fractional heat kernel. Throughout this section we fix $\alpha \in (0,1)$.  
We first collect some properties of the function $h_{\alpha}$ defined in \eqref{fracheatsgrp3} by
\[h_\alpha(t,y)=\int_{0}^{\infty}h(\tau t^{\frac{1}{\alpha}},y)f_{1,\alpha}(\tau)\dd \tau.\]

\begin{proposition}\label{propalpha}
Given $\alpha \in (0,1)$, the function $h_{\alpha}$ has the following properties:
\begin{enumerate}
	\item $h_{\alpha}\in C^{\infty}(\hat{\G}\setminus \{(0,0)\})$,
	\item $h_{\alpha}(\lambda^{2\alpha} t, \delta_{\lambda} (x))=\lambda^{-Q} h_{\alpha}(t,x)$ for every $x\in \G$ and $t,\lambda>0$,
	\item $h_{\alpha}(t,x)=0$ for every $t<0$ and $\int_{\G} h_{\alpha}(t,x) \dd x=1$ for every $t>0$,
	\item $h_{\alpha}(t,x)=h_{\alpha}(t,x^{-1})$ for every $t>0$ and $x\in \G$.
\end{enumerate}
\end{proposition}

\begin{proof}
(1), (2) and (4) are trivial consequences of Theorem \ref{heatProp} (1), (2) and (4) respectively. Property (3) follows from Theorem \ref{heatProp} (3) by observing that for $t>0$ we have
\begin{equation}
\int_{\G} h_{\alpha}(t,x) \dd x=\int_{\G}\int_{0}^{\infty}h(\tau t^{\frac{1}{\alpha}},x)f_{1,\alpha}(\tau)\dd \tau \dd x=\int_{0}^{\infty}f_{1,\alpha}(\tau) \dd \tau=1,
\end{equation}
where the last equality is \cite[Proposition 3, Chapter IX]{Yosida}. 
\end{proof}

In what follows, given $f\in L^p(\G)$ we will use the notation
\begin{align}\label{defU}
u(t,x):=(h_{\alpha}*f)(t,x)&=\int_{\G}h_{\alpha}\left(t,y\right)f\left(y^{-1}x\right) \dd y \nonumber\\
&=\int_{\G}h_{\alpha}\left(t,xy^{-1}\right)f\left(y\right) \dd y.
\end{align}

Recall that $n$ is the topological dimension of $\G$.
\begin{proposition}\label{pest}
Let $k \in \mathbb{N}$ and $\beta\in \mathbb{N}^n$. Then $h_{\alpha}(t,x)$ satisfies for $t>0$ and $x\not=0$:
\begin{equation*}
\Big|\frac{\partial ^{k}}{\partial t^{k}}D^{\beta}h_{\alpha}(t,x)\Big|\lesssim\left\{\begin{array}{ll}
|x|^{-(Q+|\beta|+2\alpha k)} & \quad \text{if } |x|^{2\alpha}\geq t,\\
t^{-\frac{Q+|\beta|+2\alpha k}{2\alpha}} & \quad \text{if } |x|^{2\alpha}\leq t.
\end{array}
\right.
\end{equation*}
Further, for $1\leq p\leq r\leq \infty$ and $\delta=Q(1/p-1/r)$, we have for all $t>0$
\[
\Big\|\frac{\partial ^{k}}{\partial t^k}D^{\beta}u(t,x)\Big\|_{L^{r}}\lesssim t^{-\frac{|\beta|+2\alpha k +\delta}{2\alpha}} \|f\|_{L^{p}}.
\]
\end{proposition}

\begin{proof}
We start with the pointwise estimates. By Proposition \ref{propalpha}(2) we have:
$$h_{\alpha}(r^{2\alpha}t,\delta_r(x))=r^{-Q}h_{\alpha}(t,x).$$
If $|x|^{2\alpha}\geq t$, we have
\begin{align}
h_{\alpha}(t,x)&=|x|^{-Q}h_{\alpha}(t|x|^{-2\alpha},\delta_{\frac{1}{|x|}}(x))\notag\\
&\leq |x|^{-Q}\sup_{\substack{|y|=1\\0<t_{0}\leq 1}}h_{\alpha}(t_{0},y)\notag.
\end{align}
Hence to conclude the case $|x|^{2\alpha}\geq t$ for the first inequality with $\beta=0$ and $k=0$, it suffices to prove that
\[\sup_{\substack{|y|=1\\0<t_{0}\leq 1}}h_{\alpha}(t_{0},y)<\infty.\]
Indeed, from the expression of $h_{\alpha}$, Theorem \ref{heatProp}, the boundedness of $h(t,y)$ on the set $|y|=1$ and the fact that $f_{1,\alpha}(\tau)$ is continuous and integrable in $\tau$ \cite[Proposition 3, Chapter IX]{Yosida} we have that
\begin{align*}
|h_{\alpha}(t_0,y)|&\leq \int_0^{\infty}|h(\tau t_0^{\frac{1}{\alpha}},y)|f_{1,\alpha}(\tau)\dd \tau\\
&\lesssim \int_0^\infty f_{1,\alpha}(\tau) \dd \tau\\
&<\infty.
\end{align*}
On the other hand, if $|x|^{2\alpha}\leq t$, then we have
\begin{align}
h_{\alpha}(t,x)&=h_{\alpha}(t,\delta_{t^{\frac{1}{2\alpha}}}\delta_{t^{-\frac{1}{2\alpha}}}(x))\notag\\
&\leq t^{-\frac{Q}{2\alpha}}\sup_{0<|y|\leq 1}h_{\alpha}(1,y).\notag
\end{align}
The thesis follows if we prove that $\sup_{0<|y|\leq 1}h_{\alpha}(1,y)<\infty$. Using again the expression of $h_{\alpha}$ and the fact that $h(\tau,y)$ is uniformly bounded if $\tau\geq 1$, we have
\begin{align*}
|h_{\alpha}(1,y)|&\leq \int_0^1 |h(\tau,y)|f_{1,\alpha}(\tau) \dd \tau + \int_1^{\infty} |h(\tau,y)|f_{1,\alpha}(\tau) \dd \tau\\
&\lesssim \int_0^1 \tau^{-\frac{Q}{2}}e^{-\frac{|y|^{2}}{c\tau}}f_{1,\alpha}(\tau) \dd \tau+ 1.
\end{align*}
By \eqref{stimaz}
$$\int_{0}^{\infty}e^{-\lambda a}f_{t,\alpha}(\lambda)d\lambda =e^{-ta^{\alpha}}.$$
Therefore, $f_{t,\alpha}$ is the density of an $\alpha$-stable subordinator. Now from \cite[Eq. 14]{Bog}, we have for $t>0$ and $\lambda>0$,
$$f_{t,\alpha}(\lambda)\lesssim t\lambda^{-1-\alpha}e^{-t\lambda^{-\alpha}}.$$
Hence,

$$\int_0^1 \tau^{-\frac{Q}{2}}e^{-\frac{|y|^{2}}{c\tau}}f_{1,\alpha}(\tau) \dd \tau \lesssim \int_{0}^{1}\tau^{-(\frac{Q}{2}+1+\alpha)}e^{-\frac{1}{\tau^{\alpha}}}\dd \tau<\infty$$
and we conclude as before. This provides us with the following estimate:
$$h_{\alpha}(t,y)\lesssim \max\left(t^{-\frac{Q}{2\alpha}},\frac{1}{|y|^{Q}}\right).$$
The proof of the pointwise estimates of the derivatives follows from the formula
$$\frac{\partial ^{k}}{\partial t^k}D^{\beta}h_{\alpha}(r^{2\alpha}t,\delta_r(x))=r^{-(Q+|\beta|+2k\alpha)}\frac{\partial ^{k}}{\partial t^k}D^{\beta}h_{\alpha}(t,x)$$
and Theorem \ref{heatProp}(6).
Let us move to the $L^{p}$ estimates. Using Young's inequalities, since $u(t,x)=(h_{\alpha}*f)(t,x)$, we have that
$$\Big\|\frac{\partial ^{k}}{\partial t^k}D^{\beta}u(t,x)\Big\|_{L^{p}}\leq \Big\|\frac{\partial ^{k}}{\partial t^k}D^{\beta}h_{\alpha}(t,\cdot)\Big\|_{L^{1}}\|f\|_{L^{p}}.$$
The conclusion follows observing that
\begin{align}
\int_{\G}\Big|\frac{\partial ^{k}}{\partial t^k}D^{\beta}h_{\alpha}(t,x)\Big|\dd x&=\int_{|x|<t^{\frac{1}{2\alpha}}}\Big|\frac{\partial ^{k}}{\partial t^k}D^{\beta}h_{\alpha}(t,x)\Big|\dd x +\int_{|x|\geq t^{\frac{1}{2\alpha}}}\Big|\frac{\partial ^{k}}{\partial t^k}D^{\beta}h_{\alpha}(t,x)\Big|\dd x\notag\\
&\lesssim \int_{|x|<t^{\frac{1}{2\alpha}}} t^{-\frac{Q+|\beta|+2\alpha k}{2\alpha}}\dd x+\int_{|x|\geq t^{\frac{1}{2\alpha}}}|x|^{-(Q+|\beta|+2\alpha k)}\dd x\notag\\
&\lesssim t^{-\frac{|\beta|+2\alpha k}{2\alpha}}+C_{Q}\int_{t^{\frac{1}{2\alpha}}}^{\infty} r^{Q-1} r^{-(Q+|\beta|+2\alpha k)}\, dr\\
&\lesssim t^{-\frac{|\beta|+2\alpha k}{2\alpha}}\notag.
\end{align}

If $q$ is chosen such that $1/r=1/p+1/q-1\geq 0$, then by Young's inequality
\begin{align*}
\Big\|\frac{\partial ^{k}}{\partial t^k}D^{\beta}u(t,x)\Big\|_{L^{r}}&\leq \Big\|\frac{\partial ^{k}}{\partial t^k}D^{\beta}h_{\alpha}(t,\cdot)\Big\|_{L^{q}}\|f\|_{L^{p}}\\
& \lesssim t^{-\frac{|\beta|+2\alpha k+Q(1-1/q)}{2\alpha}}\|f\|_{L^{p}}\\
& \lesssim t^{-\frac{|\beta|+2\alpha k+\delta}{2\alpha}}\|f\|_{L^{p}}.
\end{align*}
\end{proof}

We recall the following useful lemma \cite{mur}:

\begin{lemma}\label{mur}
Let $(S_{1},\mu_{1})$ and $(S_{2},\mu_{2})$ be $\sigma$-finite measure spaces. Fix a $\mu_{1}\times \mu_{2}$-measurable function $K$ for which there exists $C>0$ such that
\begin{enumerate}
\item $|K(x,y)|\leq C$ for $\mu_1\times\mu_2$ a.e. $(x,y)\in S_{1}\times S_{2}$,
\item $\int_{S_{1}}|K(x,y)|d\mu_{1}(x)\leq C$ for $\mu_2$ a.e. $y\in S_{2}$,
\item $\int_{S_{2}}|K(x,y)|d\mu_{2}(y)\leq C$ for $\mu_1$ a.e. $x\in S_{1}$.
\end{enumerate}
Then the integral operator defined by $T(f)=\int_{S_{2}}K(x,y)f(y)\dd \mu_{2}(y)$ is bounded from $L^{p}(S_{2},\mu_{2})$ to $L^{p}(S_{1},\mu_{1})$ for $1\leq p\leq \infty$.
\end{lemma}

Let $0<s<1$, $1\leq p \leq \infty$, $1\leq q<\infty$ and recall the function $u$ defined in \eqref{defU}. We consider the following semi-norm on Besov spaces:
$$\|f\|_{s,p,q}=\left(\int_{0}^{\infty}\left(t^{1-\frac{s}{2}}\left\|\frac{\partial u}{\partial t}(t,\cdot)\right\|_{L^p}\right)^{q}\frac{\dd t}{t}\right)^{\frac{1}{q}}.$$
We can now prove our characterization of Besov spaces using the fractional heat kernel. The following result will be crucial later.

\begin{theorem}\label{pp}
Let $0<s<1$, $0<\alpha<1$, $1\leq p \leq \infty$ and $1\leq q<\infty$. Then for any $f\in L^p(\G)$ we have
$$\|f\|_{s,p,q}\approx \|f\|_{\dot{B}^{\alpha s}_{p,q}}.$$
\end{theorem}

\begin{proof}
We will show the equivalence of these two semi-norms for $f\in \mathcal{S}(\G)$ and the result will follow then by density. By Proposition \ref{propalpha}(3), $\int_{\G}\frac{\partial h_{\alpha}}{\partial t}\dd x=0$, therefore by \eqref{defU} and Proposition \ref{propalpha}(4)
$$\frac{\partial u}{\partial t}(t,x)=\int_{\G}\frac{\partial h_{\alpha}}{\partial t}(t,y)\Big(f(y^{-1}x)-f(x)\Big)\dd y.$$
Denoting $\omega_{p}(y)=\|f(xy)-f(x)\|_{L^{p}}$ and using Minkowski's integral inequality, we get 
\begin{align*}
\Big\|\frac{\partial u}{\partial t}\Big\|_{L^{p}}&=\left(\int_{\G}\left|\int_{\G} \frac{\partial h_{\alpha}}{\partial t}(t,y) (f(xy)-f(x)) \dd y\right|^p \dd x\right)^{\frac{1}{p}}\\
&\leq \int_{\G}\Big|\frac{\partial h_{\alpha}}{\partial t}(t,y)\Big|\omega_{p}(y)\dd y.
\end{align*}
Now using Proposition \ref{pest}, we have
$$t^{1-\frac{s}{2}}\Big\|\frac{\partial u}{\partial t}\Big\|_{L^{p}} \lesssim \left(t^{1-\frac{s}{2}}\int_{|y|^{2\alpha}\geq t}|y|^{-(Q+2\alpha)}\omega_{p}(y)\dd y+t^{-\frac{Q+s\alpha}{2\alpha}}\int_{|y|^{2\alpha}\leq t}\omega_{p}(y)\dd y\right).$$
Hence,
\begin{align}
\Big(\int_{0}^{\infty}\Big(t^{1-\frac{s}{2}}\Big\|\frac{\partial u}{\partial t}\Big\|_{L^{p}}\Big)^{q}\frac{\dd t}{t}\Big)^{\frac{1}{q}}&\lesssim \Big(\int_{0}^{\infty}\Big(t^{1-\frac{s}{2}}\int_{|y|^{2\alpha}\geq t}|y|^{-(Q+2\alpha)}\omega_{p}(y)\dd y\Big)^{q}\frac{\dd t}{t}\Big)^{\frac{1}{q}}\notag \\
&\quad+\Big(\int_{0}^{\infty}\Big(t^{-\frac{Q+s\alpha}{2\alpha}}\int_{|y|^{2\alpha}\leq t}\omega_{p}(y)\dd y\Big)^{q}\frac{\dd t}{t}\Big)^{\frac{1}{q}}\notag\\
&\lesssim \Big(\int_{0}^{\infty}\Big(\int_{\G}t^{1-\frac{s}{2}}|y|^{-2\alpha+\alpha s}\chi_{|y|^{2\alpha}\geq t}(y)|y|^{-\alpha s}\omega_{p}(y)|y|^{-Q}\dd y\Big)^{q}\frac{\dd t}{t}\Big)^{\frac{1}{q}}\notag \\
&\quad+\Big(\int_{0}^{\infty}\Big(\int_{\G}t^{-\frac{Q+\alpha s}{2\alpha}}|y|^{Q+\alpha s}\chi_{|y|^{2\alpha}\leq t}(y)|y|^{-\alpha s}\omega_{p}(y)|y|^{-Q}\dd y\Big)^{q}\frac{\dd t}{t}\Big)^{\frac{1}{q}}\notag\\
&=I_{1}+I_{2}\label{split}.
\end{align}
For the integral $I_{1}$, we apply Lemma \ref{mur} with
\begin{itemize}
\item $(S_{1},\mu_{1})=((0,\infty),\frac{\dd t}{t})$,
\item $(S_{2},\mu_{2})=(\G,\frac{\dd y}{|y|^{Q}})$,
\item $K=K_{1}(t,y)=t^{1-\frac{s}{2}}|y|^{\alpha s-2\alpha} \chi_{|y|^{2\alpha}\geq t}$, 
\item $p$ replaced by $q$,
\item $f$ replaced by $\tilde{f}(y)=|y|^{-\alpha s}\omega_{p}(y)$.
\end{itemize}
It is not hard to verify that the assumptions of the lemma are satisfied with $C=2$. For instance, to verify Lemma \ref{mur}(2) we need to show $\int_{S_{1}}|K(t,y)|d\mu_{1}(t)\leq C$. To see this we compute as follows
\begin{align*}
\int_{S_{1}}|K(t,y)|d\mu_{1}(t)&=\int_{0}^{\infty}t^{1-\frac{s}{2}}|y|^{\alpha s-2\alpha}\chi_{|y|^{2\alpha}\geq t} \frac{\dd t}{t}\\
&=|y|^{\alpha s-2\alpha}\int_{0}^{|y|^{2\alpha}}t^{-\frac{s}{2}}\dd t\\
&=\frac{1}{1-\frac{s}{2}}|y|^{\alpha s-2\alpha} |y|^{2\alpha (1-\frac{s}{2})}\\
&=\frac{1}{1-\frac{s}{2}}.
\end{align*}
To see $\tilde{f}\in L^{q}(S_{2},\mu_{2})$ we notice that for $|y|>1$, $\omega_{p}(y)\leq 2\|f\|_{L^{p}}$ and hence \[|\tilde{f}(y)|^{q}\lesssim \frac{1}{|y|^{q\alpha s}} \in L^{1}(|y|>1, \mu_{2}).\] Now for $|y|<1$, since $f\in \mathcal{S}(\G)$ we have $\omega_{p}(y)\lesssim |y|\|\nabla_{\G}f\|_{L^{p}}$. Thus 
\[|\tilde{f}(y)|^{q}\lesssim  |y|^{q(1-\alpha  s)} \in L^{1}(|y|<1,\mu_{2}).\]
Applying Lemma \ref{mur} with these parameters leads to the estimate:
\begin{align}\label{stimaI1}I_{1}\lesssim \Big(\int_{\G}\frac{\omega_{p}(y)^{q}}{|y|^{Q+qs\alpha}}\dd y\Big)^{\frac{1}{q}}.\end{align}
Similarly, in order to bound $I_{2}$, we use Lemma \ref{mur} with the same measure spaces, the same function $\tilde{f}$, and $$ K=K_2(t,y)=t^{-\frac{Q+\alpha s}{2s}}|y|^{Q+\alpha s} \chi_{|y|^{2\alpha}\leq t}.$$
Once again this yields
\begin{align}\label{stimaI2}I_{2}\lesssim \Big(\int_{\G}\frac{\omega_{p}(y)^{q}}{|y|^{Q+qs\alpha}}\dd y\Big)^{\frac{1}{q}}.\end{align}
Combining \eqref{stimaI1} and \eqref{stimaI2} we get
\begin{align}
\|f\|_{s,p,q} &= \Big(\int_{0}^{\infty}\Big(t^{1-\frac{s}{2}}\Big\|\frac{\partial u}{\partial t}\Big\|_{L^{p}}\Big)^{q}\frac{\dd t}{t}\Big)^{\frac{1}{q}}\\ &\lesssim \Big(\int_{\G}\frac{\omega_{p}(y)^{q}}{|y|^{Q+qs\alpha}}\dd y\Big)^{\frac{1}{q}}\notag\\
&= \|f\|_{\dot{B}^{\alpha s}_{p,q}}.\notag
\end{align}
On the other hand, we have
$$f(xy)-f(x)=\lim_{\varepsilon\to 0}\int_{\varepsilon}^{t}\left(-\frac{\partial}{\partial r}u(r,xy)+\frac{\partial}{\partial r}u(r,x)\right)\dd r+u(t,xy)-u(t,x).$$
Hence
$$\omega_{p}(y)\leq 2\int_{0}^{t}\Big\|\frac{\partial}{\partial r} u(r,x)\Big\|_{L^{p}}\dd r+\|u(t,xy)-u(t,x)\|_{L^{p}}.$$
But we know that
$$ \|u(t,xy)-u(t,x)\|_{L^{p}}\lesssim |y|\|\nabla_{\G}u(t,x)\|_{L^{p}}.$$
By the semigroup property $u(t,\cdot)=h_{\alpha}(\frac{t}{2},\cdot)*u(\frac{t}{2},\cdot)$ \cite{Yosida} and Proposition \ref{pest}, we get for any $i=1,\ldots, m_1$:
\begin{align}\label{scvs}
\Big\|\frac{\partial }{\partial t}X_{i}u(t,\cdot)\Big\|_{L^{p}}&\lesssim \Big\|X_{i}h_{\alpha}\left(\frac{t}{2},\cdot\right)\Big\|_{L^{1}}\Big\|\frac{\partial}{\partial t}u\left(\frac{t}{2},\cdot\right)\Big\|_{L^{p}}\\
&\lesssim t^{-\frac{1}{2\alpha}}\Big\|\frac{\partial}{\partial t}u\left(\frac{t}{2},\cdot\right)\Big\|_{L^{p}}.\notag
\end{align}
Since $\|X_{i}u(t,\cdot)\|_{L^{\infty}}\to 0$ as $t\to \infty$,
 we obtain
$$X_{i}u(t,x)=-\int_{t}^{\infty}\frac{\partial }{\partial r}X_{i}u(r,x)\dd r.$$
Thus by \eqref{scvs}
\begin{align}
\|X_{i}u(t,\cdot)\|_{L^{p}}&\lesssim \int_{t}^{\infty}r^{-\frac{1}{2\alpha}}\left\|\frac{\partial}{\partial r}u\left(\frac{r}{2},\cdot\right)\right\|_{L^{p}}\dd r\notag\\
&\lesssim \int_{\frac{t}{2}}^{\infty}r^{-\frac{1}{2\alpha}}\left\|\frac{\partial}{\partial r}u(r,\cdot)\right\|_{L^{p}}\dd r.\notag
\end{align}

Therefore,
$$\omega_{p}(y)\lesssim \int_{0}^{t}\left\|\frac{\partial}{\partial r} u(r,x)\right\|_{L^{p}}\dd r+|y|\int_{t/2}^{\infty}r^{-\frac{1}{2\alpha}}\left\|\frac{\partial}{\partial r}u(r,x)\right\|_{L^{p}}\dd r.$$
So, if one takes $t=|y|^{2\alpha}$, we have that
\begin{align}
 \|f\|_{\dot{B}^{\alpha s}_{p,q}} &=\Big(\int_{\G}(|y|^{-\alpha s} \omega_{p}(y))^{q}|y|^{-Q}\dd y\Big)^{\frac{1}{q}}\\
 &\lesssim \Big( \int_{\G}(\int_{0}^{|y|^{2\alpha}}|y|^{-\alpha s}\Big\|\frac{\partial}{\partial t}u(t,x)\Big\|_{L^{p}}\dd t)^{q}|y|^{-Q}\dd y\Big)^{\frac{1}{q}}\notag\\
&\quad+\Big( \int_{\G}(\int_{|y|^{2\alpha}/2}^{\infty}|y|^{1-\alpha s}t^{\frac{-1}{2\alpha}}\Big\|\frac{\partial}{\partial t}u(t,x)\Big\|_{L^{p}}\dd t)^{q}|y|^{-Q}\dd y\Big)^{\frac{1}{q}}\notag\\
&=: I_{1}+I_{2}\notag\\
&\lesssim \Big(\int_{0}^{\infty}(t^{1-\frac{s}{2}}\Big\|\frac{\partial}{\partial t}u(t,x)\Big\|_{L^{p}})^{q}\frac{\dd t}{t}\Big)^{\frac{1}{q}}\\\notag
&=\|f\|_{s,p,q}\notag.
\end{align}
Here we used Lemma \ref{mur} to get the last inequality. Indeed, we take $\tilde{f}(t)=t^{1-\frac{s}{2}}\Big\|\frac{\partial}{\partial t}u(t,x)\Big\|_{L^{p}}$, $(S_{1},\mu_1)=((0,\infty),\frac{\dd t}{t})$ and $(S_{2},\mu_{2})=(\G,\frac{\dd y}{|y|^{Q}})$. For $I_{1}$, we consider the kernel
$$K_{1}(t,y)=|y|^{-\alpha s}t^{\frac{s}{2}}\chi_{t<|y|^{2\alpha}},$$
and for $I_{2}$ we consider the kernel
$$K_{2}(t,y)=t^{\frac{s}{2}-\frac{1}{2\alpha}}|y|^{1-\alpha s}\chi_{t>|y|^{\alpha}/2}.$$
This completes the proof.
\end{proof}

\section{Besov Spaces via Fractional Poisson Kernel}\label{fractionalpoisson}

In this section we characterize Besov spaces $B^s_{p,q}(\mathbb{G})$ using the fractional Poisson kernel. Recall \cite[Eq. 26]{Fran} that for any $0<\alpha<1$ the fractional Poisson kernel can be written as
\[p_{\alpha}(t,x)=C_{\alpha}t^{2\alpha}\int_{0}^{\infty}r^{-(1+\alpha)}e^{-\frac{t^{2}}{4r}}h(r,x) \dd r,\]
where $C_{\alpha}=(4^{\alpha}\Gamma(\alpha))^{-1}$.

\begin{proposition}\label{stimap}
Let $n\in \mathbb{N}$ and $\beta\in \mathbb{N}^n$. Then
$$
\left|D^{\beta}p_{\alpha}(t,x)\right|\lesssim \left\{\begin{array}{ll}
|x|^{-(Q+|\beta|)}\quad & \text{if} \quad |x|\geq t,\\
t^{-(Q+|\beta|)} \quad & \text{if} \quad |x|\leq t
\end{array}
\right.
$$
\end{proposition}

\begin{proof}
Using Theorem \ref{heatProp}(2) it is easy to see that
\[
p_{\alpha}(\lambda t, \delta_{\lambda}(x))=\lambda^{-Q} p_{\alpha}(t,x)\ \mbox{for every}\ t>0,\ x\in\G\ \mbox{and}\ \lambda>0.
\]  
Therefore, if $t\leq |x|$ we get
$$p_{\alpha}(t,x)=|x|^{-Q}p_{\alpha}\left(\frac{t}{|x|},\delta_{\frac{1}{|x|}}(x)\right)\leq |x|^{-Q}\sup_{\substack{0<t_{0}\leq 1\\|y|=1}}p_{\alpha}(t_{0},y).$$
Hence to estimate $|p_{\alpha}(t,x)|$ for $|x|\geq t$ it is enough to show that \[\sup_{\substack{0< t_{0}\leq 1\\|y|=1}}p_{\alpha}(t_{0},y)<\infty.\] Indeed, from the expression of $p_{\alpha}$ and Theorem \ref{heatProp} (5) we have that
\begin{align}
|p_{\alpha}(t_{0},y)|&\lesssim t_{0}^{2\alpha}\int_{0}^{\infty}r^{-(1+\alpha)}e^{-\frac{t_{0}^{2}}{4r}}h(r,y)\dd r\notag\\
&\lesssim t_{0}^{2\alpha}\Big(\int_{0}^{1}r^{-(1+\alpha)}e^{-\frac{t_{0}^{2}}{4r}}\dd r +\int_{1}^{\infty}r^{-(1+\alpha)}e^{-\frac{t_{0}^{2}}{4r}}r^{-\frac{Q}{2}}\dd r\Big)\notag\\
&\lesssim I+II\notag.
\end{align}
The integral $II$ can be easily bounded, indeed we have
$$II\lesssim \int_{1}^{\infty}r^{-(1+\alpha+\frac{Q}{2})}\dd r<\infty.$$
For the integral $I$ we have, by the change of variable $s=\frac{r}{t_{0}^{2}}$,
$$I\lesssim \int_{0}^{\frac{1}{t_{0}^{2}}}s^{-(1+\alpha)}e^{-\frac{1}{4s}}\dd s \lesssim \int_{0}^{\infty}s^{-(1+\alpha)}e^{-\frac{1}{4s}}\dd s <\infty.$$
It follows that if $t\leq |x|$, then $p_{\alpha}(t,x)\lesssim |x|^{-Q}$.
Similarly, for $|x|\leq t$, we get
$$p_{\alpha}(t,x)=t^{-Q}p_{\alpha}\left(1,\delta_{\frac{1}{t}}(x)\right)\leq t^{-Q}\sup_{|y|\leq 1}p_{\alpha}(1,y).$$
Therefore, the estimate follows, if one shows $\sup_{|y|\leq 1}p_{\alpha}(1,y)<\infty$.  Using \cite[Theorem 4]{Saka} we can estimate
\begin{align*}
	p_{\alpha}(1,y)&\lesssim \int_{0}^{|y|^{2}}r^{-(1+\alpha)}e^{-\frac{1}{4r}}|y|^{-Q}\dd r+\int_{|y|^{2}}^{\infty}r^{-(1+\alpha)}e^{-\frac{1}{4r}}r^{-\frac{Q}{2}}\dd r\\
	&=I+II.
\end{align*}
By the strong decay of the exponential term, we have that
$$II\lesssim \int_{0}^{\infty}r^{-(1+\alpha+\frac{Q}{2})}e^{-\frac{1}{4r}}\dd r<\infty.$$
In order to bound $I$, one can assume that $|y|$ is small or else the bound is obvious. For any $N\in \N$, there exists $C=C(N)>0$ such that for $|y|$ sufficiently small we have
$$e^{-\frac{1}{4r}}\leq C r^{N} \text{ for all } r\in (0, |y|^2).$$
Hence if we pick $N>\frac{1}{2}Q+\alpha$, then
$$I\lesssim |y|^{-Q}\int_{0}^{|y|^2}r^{-(1+\alpha-N)}\dd r\lesssim |y|^{-(Q+2(\alpha-N))}.$$
This gives the desired result, leading to $p_{\alpha}(t,x)\lesssim t^{-Q}$.

Finally, from the homogeneity of $h$ we have
$$D^{\beta}p_{\alpha}(rt,\delta_r(x))=r^{-(Q+|\beta|)}D^{\beta}p_{\alpha}(t,x).$$
The bounds then follow as in Proposition \ref{pest}.
\end{proof}

If one wants to consider also derivatives in $t$, the bound becomes a bit more involved. We adapt the following notation: $f\lesssim_{k,n} g$ if $\partial_{t}^{i}D^{\beta}f\lesssim\partial_{t}^{i}D^{\beta}g$ for all $i\leq k$ and $|\beta|\leq n$.
\begin{lemma}\label{lem4.2}
For all $n,k\in \N$, $t\in (0,\infty)$ and $x\in \mathbb{G}$ we have
$$p_{\alpha}(t,x)\lesssim_{k,n}\frac{t^{2\alpha}}{(t^{2}+|x|^{2})^{\frac{Q+2\alpha}{2}}}.$$
\end{lemma}

\begin{proof}
We will write the proof for $\partial_{t}p_{\alpha}$, for higher derivatives the proof follows the same strategy. From the formula defining $p_{\alpha}$ we have that
\begin{align}\partial_{t}p_{\alpha}(t,x)&\lesssim t^{2\alpha-1}\int_{0}^{\infty}r^{-(1+\alpha)}e^{-\frac{t^{2}}{4r}}h(r,x)\dd r+t^{2\alpha+1}\int_{0}^{\infty}r^{-(2+\alpha)}e^{-\frac{t^{2}}{4r}}h(r,x)\dd r\notag\\
&= I+II.
\end{align}

Let us focus on I. Using the estimate $h(r,x)\lesssim r^{-\frac{Q}{2}}e^{-\frac{|x|^{2}}{cr}}$ and a substitution of the form $s=\frac{r}{|x|^{2}}$ we have for $t<|x|$,
$$(t^{2}+|x|^{2})^{\frac{Q+2\alpha}{2}}I\lesssim|x|^{2\alpha+Q}I\lesssim t^{2\alpha-1}\int_{0}^{\infty}s^{-(1+\alpha)}e^{-\frac{t^{2}}{4|x|^{2}s}}s^{-\frac{Q}{2}}e^{-\frac{1}{cs}}ds.$$
This last integral is uniformly bounded by $\int_{0}^{\infty}s^{-(1+\alpha+\frac{Q}{2})}e^{-\frac{1}{cs}}ds$, which is finite since the integrand at infinity behaves like $s^{-(1+\alpha+\frac{Q}{2})}$ and at zero it vanishes at infinite order. Therefore,
$$I\lesssim \frac{t^{2\alpha-1}}{(t^{2}+|x|^{2})^{\frac{Q+2\alpha}{2}}}.$$
Using the same method, we see that the second integral satisfies
$$II\lesssim \frac{t^{2\alpha-1}}{(t^{2}+|x|^{2})^{\frac{Q+2\alpha}{2}}}.$$
Let $t>|x|$. From the homogeneity of $h$ we have
\[
\partial_t^k p_{\alpha}(rt,\delta_{r}x)= r^{-(Q+k)} \partial_t^k p_{\alpha}(t,x).
\] 
Therefore using exactly the same procedure as in the proof of Proposition \ref{stimap} we get
$$\partial_{t}p_{\alpha}(t,x)\lesssim \frac{1}{t^{Q+1}}$$
and this proves our estimate.
\end{proof}

As in the previous section, given $f\in \mathcal{S}(\G)$ we will denote by 
\begin{equation}\label{defPo}
u(t,x)=(p_{\alpha}*f)(t,x)=\int_{\G}p_{\alpha}\left(t,y\right)f\left(y^{-1}x\right)\dd y=\int_{\G}p_{\alpha}\left(t,xy^{-1}\right)f\left(y\right) \dd y
\end{equation}

For the reverse inequality in Theorem \ref{Oct31} we need a non-degeneracy condition, that is a Calderon type formula. In the Euclidean setting one can make use of the Fourier transform, but in a more abstract setting such as ours we need a different tool. We took inspiration from \cite{GY} where a characterization of Besov spaces using the Littlewood-Paley approach is proved. Since we are dealing with continuous versions of the decomposition rather than the classical discrete one, we adopt the following notation. Given $\psi \in L^{1}(\G)$, we denote by $\psi_{t}$ the function
\[\psi_{t}(x)=t^{-Q}\psi(\delta_{\frac{1}{t}}(x)) \qquad \mbox{for all }x\in \G \mbox{ and }t>0.\]

\begin{lemma}\label{lemmaoct22}
There exists $\psi \in L^{1}(\G)$ such that $\int_{\G}\psi=0$ and 
\begin{equation}\label{tesi2}\int_{0}^{\infty}t\psi_{t}*\partial_{t}p_{\alpha}(t)\frac{\dd t}{t}=\delta\quad \mbox{in}\quad \mathcal{S}'(\G).
\end{equation}
\end{lemma}

\begin{remark}
Equation \eqref{tesi2} is always to be interpreted to mean
\begin{align}\label{tesi}
\lim_{\varepsilon\to 0,A\to\infty}\int_{\varepsilon}^A t\psi_{t}*\partial_{t}p_{\alpha}(t)\frac{\dd t}{t}=\delta \quad \mbox{in} \quad \mathcal{S}'(\G).
\end{align}
More precisely, let $K_{A,\varepsilon}=\int_{\varepsilon}^A t\psi_{t}*\partial_{t}p_{\alpha}(t) \frac{\dd t}{t} \in \mathcal{S}'(\G)$, in the sense that $\langle K_{A,\varepsilon},f\rangle_{\mathcal{S}'}=\int_{\G}K_{A,\varepsilon}(y)f(y)\dd y$. The convergence of $K_{A,\varepsilon}\to \delta$ in $\mathcal{S}'(\G)$ means that $\langle K_{A,\varepsilon},f\rangle_{\mathcal{S}'}\to \langle \delta,f\rangle_{\mathcal{S}'}=f(0)$ for all $f\in \mathcal{S}(\G)$. Now, if we set $T_{A,\varepsilon}=K_{A,\varepsilon}*f$, then by definition of convolution of tempered distributions, we have for $\tilde{f}_{x}(y)=f(xy^{-1})$, that
$$T_{A,\varepsilon}(x)=\langle K_{A,\varepsilon},\tilde{f}_{x}\rangle_{\mathcal{S}'},$$
for all $x\in \G$. Thus, by the convergence of $K_{A,\varepsilon}\to \delta$ we have for all $f\in \mathcal{S}(\G)$ and all $x\in \G$, $$T_{A,\varepsilon}(x)\to f(x).$$
\end{remark}

\begin{proof}
Let us denote by $\{E_{\lambda}\}$ the spectral resolution of $-\Delta_{b}$ in $L^2(\G)$. 
Let $\phi\in C^{2-\alpha}([0,\infty))$ be as in \cite[Proposition 4.1]{Fran}. By \cite[Theorem 4.4]{Fran}, for any $u\in L^2(\G)$ and $t>0$ we have
\[
v(\cdot,t)=\int_0^{\infty}\phi(\theta t^{2\alpha}\lambda^{\alpha})\dd E(\lambda)u=u*p_{\alpha}(\cdot,t)
\]
where $\theta=(2\alpha)^{-2\alpha}$.
Moreover, by \cite[Proof of Theorem 4.4]{Fran}, the following formula holds:
\[
\phi(\theta t^{2\alpha}\lambda^{\alpha})=2^{-(\alpha+1)}c_{\alpha}\lambda^{\alpha/2}t^{\alpha}\theta^{1/2}\int_0^{\infty} \tau^{-(\alpha+1)}e^{-\tau\sqrt{\lambda}t}e^{-\frac{\sqrt{\lambda}t}{4\tau}} \dd \tau=H_{\alpha}(\sqrt{\lambda}t).
\]

%
Here $c_{\alpha}>0$ (for the precise expression see \cite[Proposition 4.1]{Fran}) and $H_{\alpha}:[0,\infty)\to \R$ denotes the continuous function defined as


$$H_{\alpha}(s)=2^{-(\alpha+1)}c_{\alpha}\theta^{1/2}s^{\alpha}\int_0^{\infty} \tau^{-(\alpha+1)}e^{-\tau s}e^{-\frac{s}{4\tau}} \dd \tau.$$

Therefore,
\begin{equation}\label{convs}
u*t\partial_t p_{\alpha}(\cdot,t)=\int_0^{\infty} \tilde{H}_{\alpha}(t\sqrt{\lambda})\dd E(\lambda)u
\end{equation}
where $\tilde{H}_{\alpha}(s)=sH_{\alpha}'(s)$. 

Now we want to find a continuous function $G:[0,\infty)\to \R$ such that
$$\int_{0}^{\infty}\tilde{H}_{\alpha}(t\sqrt{\lambda})G(t\sqrt{\lambda})\frac{\dd t}{t}=1,$$
which is equivalent to
$$\int_{0}^{\infty}\tilde{H}_{\alpha}(s)G(s)\frac{\dd s}{s}=1.$$
Since $\tilde{H}_{\alpha}(s)$ is continuous and not equal to the zero function, there exist an interval $I=[\frac{a}{2},2b]$ with $a>0$, where $|\tilde{H}_{\alpha}(s)|>0$ for all $s\in I$. Let $\eta$ be a smooth function supported in $[\frac{a}{2},2b]$, equals $1$ on $[a,b]$ and such that $0\leq \eta(s)\leq 1$ in $I$. For all $s\in I$ we define \[G_{1}(s)=\eta(s)\frac{s}{\tilde{H}_{\alpha}(s)}.\]
Therefore,
\[
\int_{0}^{\infty}\tilde{H}_{\alpha}(s)G_{1}(s)\frac{\dd s}{s}=\int_{a/2}^{2b} \eta(s)\dd s.
\]
Hence it suffices to take 
\[
G(s)=\frac{G_{1}(s)}{\int_{a/2}^{2b} \eta(t)\dd t}.
\]
Define $\hat{G}:[0,\infty)\to \R$ by $\hat{G}(\lambda)=G(\sqrt{\lambda})$. Since $G$ has compact support then using the results of Section 2.3, there exists $K_{\hat{G}}\in L^2(\mathbb{G})$ such that for all $u\in \mathcal{S}(\mathbb{G})$
\[
\hat{G}(-\Delta_b)u=u*K_{\hat{G}}.
\]
For every $t>0$ we define $\hat{G}^t(\lambda)=G(t\sqrt{\lambda})$, from (\ref{stimafond}) we get $K_{\hat{G}^t}(x)=t^{-Q}K_{\hat{G}}(\delta_{\frac{1}{t}}(x))$. To conclude the proof it suffices to take 
$\psi(x)=K_{\hat{G}}(x)$ which gives $\psi_t(x)=K_{\hat{G}^t}(x)$. Indeed, by \eqref{convs}, for all $u\in L^2(\mathbb{G})$
\begin{align}\label{scv12}
\int_{\varepsilon}^A u*t\partial_{t}p_{\alpha}(\cdot,t) \frac{\dd t}{t}=\int_{\varepsilon}^A\int_0^{\infty} \tilde{H}_{\alpha}(t\sqrt{\lambda})\dd E(\lambda)u\frac{\dd t}{t}.
\end{align}
Taking $u=f*\psi_t$ we get $u=\hat{G}^t((-\Delta_b))f=\int_0^{\infty} G(t\sqrt{\lambda}) \dd E(\lambda)f$. Therefore, \eqref{scv12} implies
\begin{align*}
\int_{\varepsilon}^A t\psi_t*\partial_{t}p_{\alpha}(t)*f \frac{\dd t}{t}&=\int_{\varepsilon}^{A}\int_0^{\infty} \tilde{H}_{\alpha}(t\sqrt{\lambda})G(t\sqrt{\lambda}) \dd E(\lambda)f\frac{\dd t}{t}\\
&=\int_{0}^\infty\int_{\varepsilon}^{A} \tilde{H}_{\alpha}(t\sqrt{\lambda})G(t\sqrt{\lambda})\frac{\dd t}{t}\dd E(\lambda)f
\end{align*}

Since, $\int_{\varepsilon}^{A} \tilde{H}_{\alpha}(t\sqrt{\lambda})G(t\sqrt{\lambda})\frac{\dd t}{t}\to 1$  as $A\to \infty$ and $\varepsilon\to 0$, for all $\lambda>0$, we have that
$$\lim_{A\to \infty,\varepsilon \to 0} \int_{\varepsilon}^A t\psi_t*\partial_{t}p_{\alpha}(t)*f \frac{\dd t}{t}=\int_{0}^{\infty}\dd E(\lambda)f=f.$$

Here the first equality follows from the fact that for the spectral measure $\dd E$, one has $\int f(\lambda)g(\lambda)\dd E(\lambda)=\int f(\lambda)\dd E(\lambda)\int g(\lambda)\dd E(\lambda)$ for all bounded Borel functions $f$ and $g$.
The second equality follows from Fubini's theorem. One way to see this is to localize at given functions $f$ and $g$, that is, in the form $$\langle\int_{a}^{b}\int_{0}^{\infty}G(t,\lambda)\dd E(\lambda)f,g\rangle_{L^{2}}\dd t=\int_{a}^{b}\int_{0}^{\infty}G(t,\lambda)\dd\mu_{f,g}(\lambda)\dd t$$
and then use the classical Fubini theorem for the measures $\dd \mu_{f,g}$ and $\dd t$ since now the measures are real valued, not operator valued.
The last convergence statement follows again from the fact that if $f_{n}$ converges pointwise to $f$, then the corresponding operators converge in the weak $*$-topology, that is $$\langle \int f_{n}(\lambda)\dd E(\lambda)u,v\rangle_{L^{2}}\to \langle \int f(\lambda)\dd E(\lambda)u,v\rangle_{L^{2}}$$ 
for all $u,v \in L^{2}$.

Finally, by \cite[Theorem 1, Lemma 6]{Cr}, we have that $\|\psi\|_{L^{1}}$ is finite. Moreover, notice that $G(s)=s^{2}G_{2}(s)$, where
$$G_{2}(s)=\frac{\eta(s)}{s\tilde{H}_{\alpha}(s)\int_{a/2}^{2b}\eta(t)\dd t}.$$
Notice that $G_{2}$ is well defined and smooth since it is supported away from $0$. Setting $\hat{G}_{2}(\lambda)=G_{2}(\sqrt{\lambda})$, we have that 
$$\hat{G}(-\Delta_{b})=-\Delta_{b}\hat{G}_{2}(-\Delta_{b}).$$
In particular, we have that $$K_{\hat{G}}=-\Delta_{b}K_{\hat{G}_{2}}.$$ 
Therefore, 
$$\int_{\G}\psi(x)\dd x=\int_{\G}-\Delta_{b}K_{\hat{G}_{2}}(x)\dd x=0.$$
\end{proof}

\begin{remark}
Notice that a similar construction can be also done for $t^{r}(-\Delta_{b})^{\frac{r}{2}}p_{\alpha}$ for $r\in [0,2]$. Namely, there for $r\in [0,1]$, there exists $\psi\in L^{1}(\G)$, with $\int_{\G}\psi=0$ and $\int_{0}^{\infty}t\psi_{t}*t^{r}(-\Delta_{b})^{r}p_{\alpha}(t)\frac{dt}{t}=\delta$. Indeed the spectral multiplier corresponding to $t^{r}(-\Delta_{b})^{\frac{r}{2}}p_{\alpha}$ can be written as $(\sqrt{\lambda}t)^{r}H(\sqrt{\lambda}t)$. Hence one needs to find $G$ such that 
$$\int_{0}^{\infty}s^{r}H(s)G(s)\frac{\dd s}{s}=1.$$
But then, using the same idea as before, we can pick $$G_{1}(s)=\eta(s)\frac{s^{1-r}}{H(s)}.$$ Another important remark is that if $\psi$ is chosen as in the proof of Lemma 4.3, then $-\Delta_{b}(\psi)\in L^{1}$. In particular, by real interpolation we have,
\[\|\nabla_{\G}\psi\|_{L^{1}}\lesssim \|\psi\|_{L^{1}}^{\frac{1}{2}}\|\psi\|_{S^{1}_{2}}^{\frac{1}{2}}\lesssim \|-\Delta_{b}\psi\|_{L^{1}}+\|\psi\|_{L^{1}}.\]
Hence $\nabla_{\G}\psi \in L^{1}(\G)$.
\end{remark}

\begin{theorem}\label{Oct31}
Let $f\in \mathcal{S}(\G)$ and $u$ be as in \eqref{defPo}. Then for $s\in (0,1)$ we have
$$\Big(\int_{0}^{\infty}(t^{1-s}\|\nabla_{\G}u(t,x)\|_{L^{p}})^{q}\frac{\dd t}{t}\Big)^{\frac{1}{q}}\approx \|f\|_{\dot{B}^{s}_{p,q}}.$$
Also, for $s<2\alpha<2$,
$$\Big(\int_{0}^{\infty}(t^{1-s}\|\frac{\partial}{\partial t}u(t,x)\|_{L^{p}})^{q}\frac{\dd t}{t}\Big)^{\frac{1}{q}}\approx \|f\|_{\dot{B}^{s}_{p,q}}$$
and for $s\in (0,2)$,
$$\Big(\int_{0}^{\infty}(t^{2-s}\|(-\Delta_b)u(t,x)\|_{L^{p}})^{q}\frac{\dd t}{t}\Big)^{\frac{1}{q}}\approx \|f\|_{\dot{B}^{s}_{p,q}}.$$
\end{theorem}

\begin{proof}
We will focus first on the proof of the first part. In this case one inequality is easy to prove but the opposite one needs another ingredient provided by Lemma \ref{lemmaoct22}. First, notice that the fact $h(t,x)=h(t,-x)$ implies
\[
\int_{\G} X_i h(r,x) \dd x=0\qquad \mbox{for all } i=1,\ldots, m \mbox{ and all } r\in (0,\infty).
\]
Hence using the explicit form of $p_{\alpha}$ yields
\[\int_{\G}\nabla_{\G}p_{\alpha}(t,x)\dd x=0.\]
Next we have
$$\nabla_{\G}u(t,x)=\int_{\G}\nabla_{\G}p_{\alpha}(t,y)\Big(f(xy)-f(x)\Big)\dd y.$$
Therefore,
$$\|\nabla_{\G}u\|_{L^{p}}\leq \int_{\G}|\nabla_{\G}p_{\alpha}(t,y)|\omega_{p}(y)\dd y,$$
where $\omega_{p}(y)=\|f(xy)-f(x)\|_{L^{p}}$.
Thus, by Proposition \ref{stimap}
$$t^{1-s}\|\nabla_{\G}u\|_{L^{p}}\lesssim t^{1-s}\int_{|y|\geq t}|y|^{-(Q+1)}\omega_{p}(y)\dd y+\int_{|y|<t}t^{-(Q+s)}\omega_{p}(y)\dd y.$$
Again here, we use the same trick as in (\ref{split}). Indeed, for the first integral we take $K(t,y)=\chi_{|y|>t}\frac{t^{1-s}}{|y|^{1-s}}$ on the spaces $((0.+\infty),\frac{\dd t}{t})$ and $(\G,\frac{\dd y}{|y|^{Q}})$ and $f(y)=\frac{\omega_{p}(y)}{|y|^{s}}$. For the second integral, we take $K(t,y)=\chi_{|y|<t}t^{-(Q+s)}|y|^{Q+s}$ with the same measure spaces and function $f$. We then have
$$\Big(\int_{0}^{\infty}(t^{1-s}\|\nabla_{\G}u(t,x)\|_{L^{p}})^{q}\frac{\dd t}{t}\Big)^{\frac{1}{q}}\lesssim \|f\|_{\dot{B}^{s}_{p,q}}.$$
This proves the first inequality.

For the reverse inequality, we first use Remark 4.5 to see that 
$$f(x)=\int_{0}^{\infty}(\psi_{t}*t(-\Delta_{b})^{\frac{1}{2}}p_{\alpha}*f)(x) \frac{\dd t}{t}.$$
Hence, using Young's inequality for convolutions in the first inequality, 
\begin{align}
&\|f(xy)-f(x)\|_{L^{p}_{x}}\\
&\quad \lesssim \int_{0}^{\infty}t\|\psi_{t}(xy)-\psi_{t}(x)\|_{L^{1}_{x}} \|(-\Delta_{b})^{\frac{1}{2}}u\|_{L^{p}}\frac{\dd t}{t}\notag\\
&\quad \lesssim \int_{0}^{\infty}\chi_{|y|\geq t}t\|\psi\|_{L^{1}}\|(-\Delta_{b})^{\frac{1}{2}}u\|_{L^{p}}\frac{\dd t}{t} +\int_{0}^{\infty} \chi_{|y|\leq t} \|t \nabla_{\G}\psi_{t}\|_{L^{1}}|y|\|(-\Delta_{b})^{\frac{1}{2}}u\|_{L^{p}}\frac{\dd t}{t}\notag\\
&\quad \lesssim \int_{0}^{\infty}\chi_{|y|\geq t}t\|\psi\|_{L^{1}}\|(-\Delta_{b})^{\frac{1}{2}}u\|_{L^{p}}\frac{\dd t}{t} +\int_{0}^{\infty} \chi_{|y|\leq t} \|\nabla_{\G}\psi\|_{L^{1}}|y|\|(-\Delta_{b})^{\frac{1}{2}}u\|_{L^{p}}\frac{\dd t}{t}.\notag\\
\end{align}
Since $\psi$ and $\nabla_{\G}\psi\in L^{1}(\G)$ (see Remark 4.5), we have
\begin{align}
\omega_{p}(y) &\lesssim \int_{0}^{\infty}\chi_{|y|\geq t}t\|(-\Delta_{b})^{\frac{1}{2}}u\|_{L^{p}}\frac{\dd t}{t}+\int_{0}^{\infty} \chi_{|y|\leq t}|y|\|(-\Delta_{b})^{\frac{1}{2}}u\|_{L^{p}}\frac{\dd t}{t}\notag\\
&\lesssim \int_{0}^{\infty}\chi_{|y|\geq t}t\|\nabla_{\G}u\|_{L^{p}}\frac{\dd t}{t}+\int_{0}^{\infty} \chi_{|y|\leq t}|y|\|\nabla_{\G}u\|_{L^{p}}\frac{\dd t}{t}.\notag
\end{align}
Here, we used in the second inequality the continuity of the Riesz transform from $L^{p}$ to $L^{p}$ which gives $\|(-\Delta_{b})^{\frac{1}{2}}u\|_{L^{p}} \lesssim \|\nabla_{\G}u\|_{L^{p}}$ (see \cite{BG}). Hence,
\begin{align}
\Big(\int_{\G}(\frac{\omega_{p}(y)}{|y|^{s}})^{q}\frac{\dd y}{|y|^{Q}}\Big)^{\frac{1}{q}}&\lesssim \Big(\int_{\G}(\int_{0}^{\infty}\frac{t}{|y|^{s}}\chi_{t\leq |y|}\|\nabla_{\G}u\|_{L^{p}} \frac{\dd t}{t})^{q}\frac{\dd y}{|y|^{Q}}\Big)^{\frac{1}{q}}\notag\\
&\qquad+\Big(\int_{\G}(\int_{0}^{\infty}\frac{|y|}{|y|^{s}}\chi_{t\geq |y|}\|\nabla_{\G}u\|_{L^{p}} \frac{\dd t}{t})^{q}\frac{\dd y}{|y|^{Q}}\Big)^{\frac{1}{q}}\notag\\
&\lesssim \Big(\int_{0}^{\infty}t^{q}\|\nabla_{\G}u\|_{L^{p}}^{q}(\int_{|y|\geq t}\frac{1}{|y|^{Q+s}}\dd y)^{q}\frac{\dd t}{t}\Big)^{\frac{1}{q}}\notag\\
&\qquad + \Big(\int_{0}^{\infty}\|\nabla_{\G}u\|_{L^{p}}^{q}(\int_{|y|\leq t}\frac{1}{|y|^{Q-(1-s)}}\dd y)^{q}\frac{\dd t}{t}\Big)^{\frac{1}{q}}\notag\\
&\lesssim \Big(\int_{0}^{\infty}(t^{1-s}\|\nabla_{\G} u\|_{L^{p}})^{q}\frac{\dd t}{t}\Big)^{\frac{1}{q}}.
\end{align}
using the same trick of Lemma 3.3. Let us now move to the second equivalence. We will prove the direct inequality and the reverse one works exactly as in the previous setting. The main difference in this second equivalence, is the fact that $\int_{\G}\partial_{t}p_{\alpha}(t,y)\dd y\not =0$ so we need to do few more manipulations in order to have a similar setting as before. First recall from \cite{Fran}, that $p_{\alpha}$ satisfies the equation $$\partial_{t}(t^{1-2\alpha}\partial_{t}p_{\alpha})+t^{1-2\alpha}\Delta_{b}p_{\alpha}=0; \text{ for } t>0.$$ Hence, we have
$$\int_{\G}\partial_{t}(t^{1-2\alpha}\partial_{t}p_{\alpha})(t,y)\dd y=0.$$
Thus,
\begin{align}
t^{1-2\alpha}\partial_{t}u(t,x)&=\int_{\G}t^{1-2\alpha}\partial_{t}p_{\alpha}(t,y)f(xy)\dd y\notag\\
&=-\int_{\G}\int_{t}^{\infty}\partial_{r}(r^{1-2\alpha}\partial_{r}p_{\alpha})(r,y)(f(xy)-f(x))\dd r \dd y\notag\\
&=\int_{\G}\int_{t}^{\infty}r^{1-2\alpha}\Delta_{b}p_{\alpha}(r,y)(f(xy)-f(x))\dd r \dd y.
\end{align}
It follows then, that
\begin{align}
t^{1-2\alpha}\|\partial_{t}u\|_{L^{p}}&\lesssim \int_{\G}\int_{t}^{\infty}r^{1-2\alpha}|\Delta_{b}p_{\alpha}(r,y)|\omega_{p}(y)\dd r \dd y\notag\\
&\lesssim \int_{|y|<t}\int_{t}^{\infty}r^{1-2\alpha}|\Delta_{b}p_{\alpha}(r,y)|\omega_{p}(y)\dd r \dd y\\
&\qquad +\int_{|y|>t}\int_{t}^{\infty}r^{1-2\alpha}|\Delta_{b}p_{\alpha}(r,y)|\omega_{p}(y)\dd r \dd y\notag\\
&=I+II.
\end{align}
We first estimate $I$. Using the fact that $|y|<t$, we have from Lemma \ref{lem4.2}, that
$$r^{1-2\alpha}|\Delta_{b}p_{\alpha}(r,y)|\lesssim \frac{1}{r^{Q+1+2\alpha}}.$$
Thus, $$I\lesssim \int_{|y|<t}\frac{1}{t^{Q+2\alpha}}\omega_{p}(y)\dd y.$$
In particular, 
$$\int_{0}^{\infty}\Big[t^{2\alpha-s}I\Big]^{q}\frac{\dd t}{t}\lesssim \int_{0}^{\infty}\Big[\int_{\G}\chi_{|y|<t}\frac{1}{t^{Q+s}}\omega_{p}(y)\dd y\Big]^{q}\frac{\dd t}{t}.$$
We use then Lemma 3.3 with the same measure spaces as before, for the kernel $K(t,y)=\chi_{|y|<t}\frac{|y|^{Q+s}}{t^{Q+s}}$ and $f(y)=\frac{\omega_{p}(y)}{|y|^{s}}$. This leads to
$$\Big(\int_{0}^{\infty}\Big[t^{2\alpha-s}I\Big]^{q}\frac{\dd t}{t}\Big)^{\frac{1}{q}}\lesssim \|f\|_{\dot{B}^{s}_{p,q}}.$$
We move now to the second term. Indeed, we have
\begin{align}
II&\lesssim \int_{|y|>t}\int_{0}^{|y|}r^{1-2\alpha}|\Delta_{b}p_{\alpha}(r,y)|\omega_{p}(y)\dd r \dd y\notag\\
&\qquad +\int_{|y|>t}\int_{|y|}^{\infty}r^{1-2\alpha}|\Delta_{b}p_{\alpha}(r,y)|\omega_{p}(y)\dd r \dd y.
\end{align}
Now using again Lemma \ref{lem4.2}, we have that, for $r<|y|$,
$$r^{1-2\alpha}|\Delta_{b}p_{\alpha}(r,y)|\lesssim \frac{r^{1-2\alpha}}{|y|^{Q+2}}.$$
Therefore, since $\alpha<1$, we have
$$\int_{0}^{|y|}r^{1-2\alpha}|\Delta_{b}p_{\alpha}(r,y)|\dd r \lesssim \frac{1}{|y|^{Q+2\alpha}}.$$
Similarly, when $r>|y|$, using Lemma \ref{lem4.2}, we have
\begin{align}
\int_{|y|}^{\infty}r^{1-2\alpha}|\Delta_{b}p_{\alpha}(r,y)|\dd r&\lesssim \int_{|y|}^{\infty}r^{1-2\alpha}\frac{r^{2\alpha}}{(r^{2}+|y|^{2})^{\frac{Q+2\alpha+2}{2}}}\dd r\notag\\
&\lesssim \int_{|y|}^{\infty}\frac{1}{r^{Q+2\alpha+1}}\dd r\notag\\
&\lesssim  \frac{1}{|y|^{Q+2\alpha}}.
\end{align}
Thus, we have
$$t^{2\alpha-s}II\lesssim \int_{\G}\chi_{|y|>t}\frac{t^{2\alpha-s}}{|y|^{Q+2\alpha}}\omega_{p}(y)\dd y.$$
We use now Lemma 3.3 with the same measure spaces as before and $p=q$, for $K(t,y)=\chi_{|y|>t}\frac{t^{2\alpha-s}}{|y|^{2\alpha-s}}$ and $f(y)=\frac{\omega_{p}(y)}{|y|^{s}}$, keeping in mind that the assumptions of Lemma 3.3 hold when $2\alpha>s$, we have
$$\Big(\int_{0}^{\infty}\Big[t^{2\alpha-s}II\Big]^{q}\frac{\dd t}{t}\Big)^{\frac{1}{q}}\lesssim \|f\|_{\dot{B}^{s}_{p,q}}.$$
Therefore, we conclude that
$$\Big(\int_{0}^{\infty}\Big[t^{1-s}\|\partial_{t}u\|_{L^{p}}\Big]^{q}\frac{\dd t}{t}\Big)^{\frac{1}{q}}\lesssim \|f\|_{\dot{B}^{s}_{p,q}}.$$
The proof of the last equivalence, is exactly similar to the first one, hence we omit it.
\end{proof}

\section{Square Function and BMO Bounds}\label{bounds}

\subsection{Square Function Bounds}

The following square function bounds using the Sobolev norm will be useful later in the applications.

\begin{theorem}\label{Oct23}
Let $f\in \mathcal{S}(\G)$ and $1<p<\infty$. Then:

For $-Q<s<1$ we have
$$\left\|\Big(\int_{0}^{\infty}[t^{1-s}|\nabla_{\mathbb{G}}u(t,x)|]^{2}\frac{\dd t}{t}\Big)^{\frac{1}{2}}\right\|_{L^{p}} \lesssim \|(-\Delta_{b})^{\frac{s}{2}}f\|_{L^{p}}.$$

For $s<2\alpha$ we have
$$\left\| \Big(\int_{0}^{\infty}[t^{1-s}|\frac{\partial}{\partial t}u(t,x)|]^{2}\frac{\dd t}{t}\Big)^{\frac{1}{2}}\right\|_{L^{p}}\lesssim \|(-\Delta_{b})^{\frac{s}{2}}f\|_{L^{p}}.$$

For $-Q<s<2$ we have
$$\left\| \Big(\int_{0}^{\infty}[t^{2-s}|\nabla_{\mathbb{G}}^{2}u(t,x)|]^{2}\frac{\dd t}{t}\Big)^{\frac{1}{2}}\right\|_{L^{p}}\lesssim \|(-\Delta_{b})^{\frac{s}{2}}f\|_{L^{p}}.$$

\end{theorem}

In order to proceed with the proof of Theorem \ref{Oct23}, we first need to recall a few important properties of square functions. For further details, we refer the reader to \cite{FS,S}.

Let $\phi\in \mathcal{S}(\G)$ be such that $\int_{\G}\phi \dd x=0$ and $\phi_{t}(x)=t^{-Q}\phi(\delta_{\frac{1}{t}}x)$. Then we define the square functions
\begin{equation}\label{squareS}
S_{\phi}^{\beta}f(x):=\Big(\int_{0}^{\infty}\int_{|x^{-1}y|<\beta t}|f*\phi_{t}(y)|^{2}t^{-Q-1}\dd y\dd t\Big)^{\frac{1}{2}}
\end{equation}
where $\beta>0$ and
\begin{equation}\label{squareg}
g_{\phi}f(x):=\Big(\int_{0}^{\infty}|f*\phi_{t}|^{2}\frac{\dd t}{t}\Big)^{\frac{1}{2}}.
\end{equation}
In \cite{FS}, the authors proved the $L^{p}$ boundedness of these operators. More precisely, they show that for $0<p<\infty$, $g_{\phi}$ and $S_{\phi}^{\beta}$ are bounded from the Hardy space $H^{p}(\G)$ to $L^{p}(\G)$.  From now on, we will write $S_{\phi}$ for $S_{\phi}^{\beta}$. In order to use this result, we need to relax the assumption $\phi\in \mathcal{S}(\G)$. We will need these bounds for some specific functions $\phi$ which are not in $\mathcal{S}(\G)$.

\begin{proposition}\label{propphit}
Let $s>-1$ and $\phi(x)=\nabla_{\mathbb{G}}(-\Delta_{b})^{\frac{s}{2}}p_{\alpha}(1,x)$ which implies the formula $\phi_{t}(x)=t^{1+s}\nabla_{\mathbb{G}}(-\Delta_{b})^{\frac{s}{2}}p_{\alpha}(t,x)$.
Then the function $K_{a}^{b}$ defined by
$$K_{a}^{b}(x)=\int_{a}^{b}\phi_{t}*\phi_{t}(x) \frac{\dd t}{t}$$
converges as $a\to 0$ and $b\to \infty$ to a function $K$ in $\mathcal{S}'(\G)$ that is smooth on $\G\setminus \{0\}$ and homogeneous of degree $-Q$ around zero.
\end{proposition}

\begin{proof}
The formula $\phi_{t}(x)=t^{1+s}\nabla_{\mathbb{G}}(-\Delta_{b})^{\frac{s}{2}}p_{\alpha}(t,x)$ follows from the homogeneity of $p_{\alpha}$. Next notice that $p_{\alpha}*p_{\alpha}(t,x)=t^{-Q}p_{\alpha}*p_{\alpha}(1,\frac{x}{t})$. Indeed, this follows from the property that $f_{t}*g_{t}=(f*g)_{t}$. Here the convolution is only on the $x$ variable, while the scaling is in the $t$ variable). Now notice that
\[K_{a}^{b}(x)=\int_{a}^{b}t^{1+2s} (-\Delta_{b})^{s+1}(p_{\alpha}*p_{\alpha})(t,x)\dd t.\]
But then one can see, using Proposition 4.1 and an interpolation inequality of the form
\[\|(-\Delta_{b})^{1+s}u\|_{L^{\infty}(\Omega)}\lesssim \|\Delta_{b}u\|_{L^{\infty}(\Omega)}^{\theta}\|(-\Delta_{b})^{2}u\|_{L^{\infty}(\Omega)}^{1-\theta}\]
with $\Omega =\{R\leq |x|\leq 2R\}$, that
$$|(-\Delta_{b})^{s+1}p_{\alpha}*p_{\alpha}|\lesssim \left\{\begin{array}{ll}
|x|^{-(Q+2s+2)} &\quad \text{if} \quad  t\leq |x|\\
t^{-(Q+2s+2)} &\quad \text{if}\quad  |x|\leq t.
\end{array}
\right.$$
Therefore, for $|x|>0$, $t^{1+2s}(-\Delta_{b})^{s+1}p_{\alpha}*p_{\alpha}(t,x)=O(t^{1+2s})$ near zero and $t^{1+2s}(-\Delta_{b})^{s+1}p_{\alpha}*p_{\alpha}(t,x)=O(t^{-Q-1})$ near $\infty$. 
Thus, as long as $1+2s>-1$, the integral converges absolutely to a smooth function on $\G\setminus\{0\}$. Moreover, if we let $K=\lim_{a\to 0;b\to \infty}K_{a}^{b}$, we have that $K(rx)=r^{-Q}K(x)$ which finishes the proof. 
\end{proof}

One also has the same result for 
\begin{equation}\label{phit}
\phi_{t}=\left\{\begin{array}{ll}
t^{1+s}\nabla_{\mathbb{G}}(-\Delta_{b})^{\frac{s}{2}}p_{\alpha} & \quad \text{if} \quad s>-1,\\
t^{1+s}(-\Delta_{b})^{\frac{s}{2}}\frac{\partial }{\partial t}p_{\alpha} &\quad \text{if} \quad s>-2\alpha,\\
t^{2+s}(-\Delta)^{\frac{s}{2}}(-\Delta_{b})p_{\alpha} &\quad \text{if} \quad s>-2,\\
t^{2+s}\nabla_{\mathbb{G}}(-\Delta_{b})^{\frac{s}{2}}\frac{\partial }{\partial t}p_{\alpha} &\quad \text{if} \quad s>-1-2\alpha.
\end{array}
\right.
\end{equation}
Recall the square functions $g_{\phi}$ and $S_{\phi}^{1}=S_{\phi}$ defined in \eqref{squareS} and \eqref{squareg}.

\begin{proposition}\label{sg}
Let $\phi_{t}$ defined in (\ref{phit}). Then $S_{\phi}$ and $g_{\phi}$ are bounded from $L^{p}$ to $L^{p}$, for $1<p<\infty$.
\end{proposition}

\begin{proof}
We will follow here the proof in \cite{FS} for the case $\phi\in \mathcal{S}(\G)$ and we will present it for $\phi_{t}$ as in Proposition \ref{propphit} since the proof is similar for the remaining functions in (\ref{phit}). Indeed, one first proves the $L^{2}$ bound, that is 
\begin{align}
\|g_{\phi}f\|^{2}_{L^{2}}&=\int_{\G}\int_{0}^{\infty}f*\phi_{t} f*\phi_{t} \frac{\dd t}{t}\notag\\
&=\int_{\G}f*K(x)f(x)\dd x\leq \|K*f\|_{L^{2}}\|f\|_{L^{2}}.
\end{align}
But, from Proposition 3.2, $K$ is a kernel of type $(0,2)$, thus we have that
$$\|K*f\|_{L^{2}}\lesssim \|f\|_{L^{2}}.$$
Therefore $$\|g_{\phi}\|_{L^{2}}\lesssim \|f\|_{L^{2}}.$$
We define the space $X=L^{2}((0,\infty),\frac{\dd t}{t})$ and the $X-$valued distribution $\Phi$ defined for $f\in \mathcal{S}(\G)$ by
$$\langle \Phi,f\rangle (t)=\int_{G}f(x)\phi_{t}(x)\dd x.$$
We claim that this distribution is well defined. Indeed, we have
$$|\langle \Phi,f\rangle (t)|\leq \frac{1}{t^{Q}}\|\phi\|_{L^{\infty}}\|f\|_{L^{1}}.$$
Next, we notice that since $\int \phi =0$ we have that
$$|\langle \Phi,f\rangle (t)|\leq \int_{\G}|f(tx)-f(0)|\phi(x)\dd x.$$
Since $f\in \mathcal{S}(\G)$ we see that 
\[t\mapsto \Big|\frac{\langle \Phi,f\rangle (t)}{t}\Big|\]
is bounded near zero, therefore $\langle \Phi,f\rangle \in L^{2}((0,\infty),\frac{\dd t}{t})$. Hence, $g_{\phi}f(x)$ is well defined and 
$$g_{\phi}f(x)=\|f*\Phi\|_{X}.$$
And so far we have proved that $g_{\phi}$ is bounded from $L^{2}$ to $L^{2}_{X}$.  Moreover, if we look at $\Phi(x)(t)=\phi_{t}(x)$ we have that
\begin{align}
\|D^{\beta}\Phi(x)\|_{X}^{2}&=\int_{0}^{\infty}|t^{1+s}D^{\beta+1}(-\Delta_{b})^{\frac{s}{2}}p_{\alpha}(t,x)|^{2}\frac{\dd t}{t}\notag\\
&\lesssim \int_{0}^{|x|}t^{1+2s}|x|^{-2(Q+\beta+s+1)}\dd t+\int_{|x|}^{\infty}t^{1+2s}t^{-2(Q+\beta+1+s)}\dd t\notag\\
&\lesssim |x|^{-2(Q+\beta)}\notag
\end{align}
Hence $\Phi$ is an $X$-valued kernel of type $(0,r)$ for al $r>0$ which leads to the boundedness of $f*\Phi$ from $L^{p}$ to $L^{p}_{X}$ for $1<p<\infty$. Thus
$$\|g_{\phi}f\|_{L^{p}}\lesssim \|f\|_{L^{p}}.$$
A similar bound holds for the operator $S_{\phi}^{\beta}$.
\end{proof}

\begin{proof}[Proof of Theorem \ref{Oct23}]
The proof of Theorem \ref{Oct23} now is a straightforward consequence of Proposition \ref{sg}. First, we write 
$$t^{1-s}\nabla_{\G}u(t,x)=t^{1-s}\nabla_{\mathbb{G}}p_{\alpha}*f=t^{1-s}\nabla_{\mathbb{G}}(-\Delta_{b})^{-\frac{s}{2}}p_{\alpha}*(-\Delta_{b})^{\frac{s}{2}}f.$$
Applying Proposition \ref{sg}, we have the desired result for $ s<1$. 
\end{proof}

\subsection{BMO Bounds}

Next we provide some equivalent characterizations of the BMO norm that will be useful in the coming applications. First, given a function $f\in L^{1}_{loc}(\G)$ and $B$ a ball in $\G$, we define
$$m_{B}=\frac{1}{|B|}\int_{B}f(x)\dd x.$$
Let $\mathcal{B}$ be the collection of all the balls in $\G$. A function $f$ is said to be in $BMO$ if 
$$[f]_{BMO}=\sup_{B\in \mathcal{B}}\frac{1}{|B|}\int_{B}|f(x)-m_{B}|\dd x<\infty.$$
We recall next the characterization of the BMO norm using the Carleson measure. If we denote by $$T(B_{r}(x_{0}))=\{(t,x)\in \R^{+}\times \G: |x_{0}^{-1}x|<r-t\},$$
then we have the following proposition \cite{S}.

\begin{proposition}
Let $f\in \mathcal{S}(\G)$ and $\phi \in \mathcal{S}(\G)$ be such that $\int_{\G}\phi(x)\dd x=0$. Then
$$\sup_{B\in \mathcal{B}}\frac{1}{|B|}\int_{T(B)}|f*\phi_{t}|^{2}\frac{\dd t\dd x}{t}\lesssim [f]_{BMO}^{2}.$$
If we assume the existence of $\psi \in \mathcal{S}(\G)$, with $\int_{\G}\psi \dd x=0$ and $\int_{0}^{\infty}\phi_{t}*\psi_{t}\frac{\dd t}{t}=\delta_{0}$, then
$$\sup_{B\in \mathcal{B}}\left(\frac{1}{|B|}\int_{T(B)}|f*\phi_{t}|^{2}\frac{\dd t\dd x}{t}\right)^{\frac{1}{2}}\approx [f]_{BMO}.$$
\end{proposition}

From the equivalence stated above, one gets the following proposition.
\begin{proposition}\label{bmo}
Let $f\in \mathcal{S}(\G)$. Then for $\phi_{t}$ defined in (\ref{phit}), we have
$$[f]_{BMO}\approx \sup_{B\in \mathcal{B}}\Big(\frac{1}{|B|}\int_{T(B)}|f*\phi_{t}|^{2}\frac{\dd t}{t}\Big)^{\frac{1}{2}}.$$
\end{proposition}

From the previous proposition we get in particular:

\begin{align}
[f]_{BMO}&\approx\sup_{B\in \mathcal{B}}\Big(\frac{1}{|B|}\int_{T(B)}|t\nabla_{\mathbb{G}}u(t,x)|^{2}\frac{\dd t}{t}\Big)^{\frac{1}{2}}\notag\\
&\approx \sup_{B\in \mathcal{B}}\Big(\frac{1}{|B|}\int_{T(B)}|t^{2}\Delta_{\mathbb{G}}u(t,x)|^{2}\frac{\dd t}{t}\Big)^{\frac{1}{2}}\notag\\
&\approx\sup_{B\in \mathcal{B}}\Big(\frac{1}{|B|}\int_{T(B)}|t\frac{\partial}{\partial t}u(t,x)|^{2}\frac{\dd t}{t}\Big)^{\frac{1}{2}},\notag
\end{align}
and in the fractional setting:
\begin{align}
[f]_{BMO}&\approx \sup_{B\in \mathcal{B}}\Big(\frac{1}{|B|}\int_{T(B)}|t^{s}(-\Delta_{b})^{\frac{s}{2}}u(t,x)|^{2}\frac{\dd t}{t}\Big)^{\frac{1}{2}}\notag\\
&\approx \sup_{B\in \mathcal{B}}\Big(\frac{1}{|B|}\int_{T(B)}|t^{1+s}\nabla_{\mathbb{G}}(-\Delta_{b})^{\frac{s}{2}} u(t,x)|^{2}\frac{\dd t}{t}\Big)^{\frac{1}{2}}.
\end{align}

We finish now by recalling the following duality result between the Carleson measure and the square function \cite{S}.

\begin{lemma} Let $G,F:\R^{+}\times \G\to \R$ be two functions. Then
\begin{align*}
&\int_{\G}\int_{0}^{\infty}F(t,x)G(t,x)\frac{\dd t}{t}\dd x\\
&\qquad \lesssim \sup_{B\in\mathcal{B}}\Big(\frac{1}{|B|}\int_{T(B)}|F(t,y)|^{2}\frac{\dd t}{t}\dd y\Big)^{\frac{1}{2}}\int_{\G}\Big(\int_{|y^{-1}x|<t}|G(t,y)|^{2}\frac{\dd t}{t^{Q+1}}\dd y\Big)^{\frac{1}{2}}\dd x,
\end{align*}
whenever the right hand side is finite.
\end{lemma}

\begin{corollary}\label{cor4}
Let $G:\R^{+}\times \G\to \R$ such that
\[\int_{\G}\Big(\int_{|y^{-1}x|<t}|G(t,y)|^{2}\frac{\dd t}{t^{Q+1}}\dd y\Big)^{\frac{1}{2}}\dd x<\infty.\]
Then
$$\int_{\G}\int_{0}^{\infty}(t\frac{\partial}{\partial t}u(t,x))G(t,x)\frac{\dd t}{t}\dd x\lesssim [f]_{BMO}\int_{\G}\Big(\int_{|y^{-1}x|<t}|G(t,y)|^{2}\frac{\dd t}{t^{Q+1}}\dd y\Big)^{\frac{1}{2}}\dd x.$$

This inequality still holds if we replace $t\frac{\partial}{\partial t}u(t,x)$ by any one of:
\begin{itemize}
\item $t\nabla_{\G}u(t,x)$,
\item $t^{2}\Delta_{b} u(t,x)$, 
\item $t^{s}(-\Delta_{b})^{\frac{s}{2}}u(t,x)$,
\item $t^{1+s}\tilde{\nabla}(-\Delta_{b})^{\frac{s}{2}} u(t,x)$.
\end{itemize}
\end{corollary}

\section{Applications}\label{applications}

Before starting this section we recall some relevant maximal functions bounds. Given a function $\phi\colon \G \to \R$ satisfying the growth condition
$$|\phi(x)|\lesssim \frac{1}{(1+|x|)^{\lambda}},$$
for a given $\lambda>0$, one can define the two maximal functions
$$(\mathcal{M}_{\phi}^{0}f)(x)=\sup_{t>0} (f*\phi_{t})(x)$$
and
$$(\mathcal{M}_{\phi}f)(x)=\sup\{|f*\phi_{t}|: |x^{-1}y|<t, 0<t<\infty\}.$$
With these definitions, one has the following theorem \cite{FS}.

\begin{theorem}\label{thmb}
For $\lambda>Q$, $\mathcal{M}_{\phi}^{0}$ and $\mathcal{M}_{\phi}$ are bounded from $L^{p}(\G)$ to $L^{p}(\G)$ for $p>1$ and from $L^{1}(\G)$ to weak $L^{1}(\G)$.
\end{theorem}

In this section if $f\in \mathcal{S}(\G)$ we will write $F_{\alpha}=f*p_{\alpha}$. We also use the notation $\tilde{\nabla}=\nabla_{\G}\oplus \frac{\partial}{\partial t}$, defined by $\tilde{\nabla}u=(\nabla_{\G}u,\frac{\partial u}{\partial t})$. In what follows, we will write $\mathcal{M}$ and $\mathcal{M}^{0}$ instead of $\mathcal{M}_{\phi}$ and $\mathcal{M}_{\phi}^{0}$, since the function $\phi$ will be different depending on the situation.

\subsection{Integral Inequalities}

\begin{theorem}\label{thmest1}
Let $f, g, h\in \mathcal{S}(\G)$ and $1<p_{1},p_{2},p_{3}<\infty$ such that $\frac{1}{p_{1}}+\frac{1}{p_{2}}+\frac{1}{p_{3}}=1$. Then 
\begin{enumerate}
\item For $s_{1},s_{2}\in (0,\min\{1,2\alpha\})$ and $Q>s_{3}\geq 0$ 
$$\int_{\R^{+}\times \G}t^{2-s_{1}-s_{2}+s_{3}}|\tilde{\nabla}F_{\alpha}||\tilde{\nabla}G_{\alpha}||H_{\alpha}|\frac{\dd x\dd t}{t}\lesssim \|(-\Delta_{b})^{\frac{s_{1}}{2}}f\|_{L^{p_{1}}}\|(-\Delta_{b})^{\frac{s_{2}}{2}}g\|_{L^{p_{2}}}\|I_{s_{3}}h\|_{L^{p_{3}}}.$$
\item For $s_{1}\in (0,\min\{1,2\alpha\})$ and $Q>s_{3},s_{2}\geq 0$ 
$$\int_{\R^{+}\times \G}t^{2-s_{1}+s_{2}+s_{3}}|\tilde{\nabla}F_{\alpha}||\tilde{\nabla}G_{\alpha}||H_{\alpha}|\frac{\dd x\dd t}{t}\lesssim \|(-\Delta_{b})^{\frac{s_{1}}{2}}f\|_{L^{p_{1}}}\|I_{s_{2}}g\|_{L^{p_{2}}}\|I_{s_{3}}h\|_{L^{p_{3}}}.$$
\item For $s_{1}\in (0,\min\{1+2\alpha,2\})$, $s_{2}\in (0,\min\{1,2\alpha\})$ and $0\leq s_{3}<Q$,  
$$\int_{\R^{+}\times \G}t^{3-s_{1}+s_{2}+s_{3}}|\nabla_{\G}\tilde{\nabla}F_{\alpha}||\tilde{\nabla}G_{\alpha}||H_{\alpha}|\frac{\dd x\dd t}{t}\lesssim \|(-\Delta_{b})^{\frac{s_{1}}{2}}f\|_{L^{p_{1}}}\|(-\Delta_{b})^{\frac{s_{2}}{2}}g\|_{L^{p_{2}}}\|I_{s_{3}}h\|_{L^{p_{3}}}.$$
\end{enumerate}
Here $I_{\alpha}$ is the fractional integration of order $\alpha$, that is $I_{\alpha}u=(-\Delta_{b})^{\frac{\alpha}{2}}u$.
\end{theorem}

\begin{proof}
We will present the proof of (1). The proofs of (2) and (3) follow the same idea. We will apply the result of Theorem \ref{thmb} for $\phi_{t}=t^{s}(-\Delta_{b})^{\frac{s}{2}}p_{\alpha}$. Indeed, we have that
$$(-\Delta_{b})^{\frac{s}{2}}p_{\alpha}(1,x)\lesssim \frac{1}{(1+|x|)^{Q+s}}.$$
Hence, we have by definition of $\mathcal{M}_{\phi}^{0}$, that 
$$\sup_{t>0}t^{s}|H_{\alpha}|=\sup_{t>0} |t^{s}(-\Delta_{b})^{\frac{s}{2}}p_{\alpha}*I_{s}f|=\mathcal{M}_{\phi}^{0}(I_{s}f).$$
Hence,
$$
\int_{\G}\int_{0}^{\infty}t^{2-s_{1}-s_{2}+s_{3}}|\tilde{\nabla}F_{\alpha}||\tilde{\nabla}G_{\alpha}||H_{\alpha}|\frac{\dd x\dd t}{t}\lesssim\int_{\G}\mathcal{M}_{\phi}^{0}(I_{s_{3}}h)\int_{0}^{\infty} t^{1-s_{1}}|\tilde{\nabla}F_{\alpha}|t^{1-s_{2}}|\tilde{\nabla}G_{\alpha}| \frac{\dd t}{t}\dd x.$$
The result then follows from H\"{o}lder's inequality and Theorem \ref{Oct23}.
\end{proof}

\begin{theorem}\label{thmest2}
Let $f, g, h \in \mathcal{S}(\G)$, and $\frac{1}{p}+\frac{1}{q}=1$ with $1<p<\infty$. Then for $s\in(0,1)$
$$\int_{\R^{+}\times \G}t^{2+2(1-s)}|\tilde{\nabla}H_{\alpha}||\nabla_{\G}\tilde{\nabla}F_{\alpha}||\tilde{\nabla}G_{\alpha}|\frac{\dd x\dd t}{t}\lesssim [h]_{BMO}\|(-\Delta_{b})^{\frac{s}{2}}f\|_{L^{p}}\|(-\Delta_{b})^{\frac{s}{2}}g\|_{L^{q}},$$
and for $s<2\alpha$ we have
$$\int_{\R^{+}\times \G}t^{2-s}|\tilde{\nabla}H_{\alpha}||\frac{\partial}{\partial t}F_{\alpha}||G_{\alpha}|\frac{\dd x\dd t}{t}\lesssim [h]_{BMO}\|(-\Delta_{b})^{\frac{s}{2}}f\|_{L^{p}}\|g\|_{L^{q}}.$$
\end{theorem}

\begin{proof}
Let us again start by proving the first claim. Indeed, using Corollary \ref{cor4} we have that
\begin{align}
&\int_{\R^{+}\times \G}t^{2+2(1-s)}|\tilde{\nabla}H_{\alpha}||\nabla_{\G}\tilde{\nabla}F_{\alpha}||\tilde{\nabla}G_{\alpha}|\frac{\dd x\dd t}{t}\\
&\qquad \lesssim[h]_{BMO}\int_{\G}\Big(\int_{|y^{-1}x|<t}\Big(t^{1+2(1-s)}|\nabla_{\G}\tilde{\nabla}F_{\alpha}||\tilde{\nabla}G_{\alpha}|\Big)^{2}\frac{\dd t}{t^{Q+1}}\Big)^{\frac{1}{2}}\dd x\notag\\
&\qquad \lesssim[h]_{BMO}\int_{\G}\mathcal{M}((-\Delta_{b})^{\frac{s}{2}}f)(x)S_{\phi}^{1}((-\Delta_{b})^{\frac{s}{2}}g)(x)\dd x\notag\\
&\qquad \lesssim [h]_{BMO}\|(-\Delta_{b})^{\frac{s}{2}}f\|_{L^{p}}\|(-\Delta_{b})^{\frac{s}{2}}g\|_{L^{q}}\notag
\end{align}
A similar proof holds for the second claim.
\end{proof}

\subsection{Three Terms Commutator}

Let $u,v\in \mathcal{S}(\G)$. Then we define the three commutator $\mathcal{H}_{\alpha}(u,v)$ by:
$$\mathcal{H}_{\alpha}(u,v):=(-\Delta_{b})^{\alpha}(uv)-u(-\Delta_{b})^{\alpha}v-v(-\Delta_{b})^{\alpha}u.$$
This commutator was studied in the Euclidean setting in \cite{lenz,Schik} and in the case of Carnot groups in \cite{M}. We want also to point out that one can obtain easy bounds for this commutator in Besov spaces. Indeed

\begin{proposition}
Let $\alpha\in (0,1)$, assume that $\frac{1}{r}=\frac{1}{p_{1}}+\frac{1}{p_{2}}$ and $\frac{1}{q_{1}}+\frac{1}{q_{2}}=1$. We let $s_{1},s_{2}>0$ so that $s_{1}+s_{2}=2\alpha$. Then we have
$$\|\mathcal{H}_{\alpha}(u,v)\|_{L^{r}}\lesssim \|u\|_{\dot{B}_{p_{1},q_{1}}^{s_{1}}}\|v\|_{\dot{B}_{p_{2},q_{2}}^{s_{2}}}.$$
\end{proposition}

\begin{proof}
The proof follows directly from the pointwise expression of the commutator. Using
$$(-\Delta_{b})^{\alpha}u(x)=\int_{\G}(u(x)-u(y))\tilde{R}_{\alpha}(xy^{-1})\dd y,$$
we can write as in \cite{M},
$$\mathcal{H}_{\alpha}(u,v)(x)=\int_{\G}[u(xy)-u(x)][v(xy)-v(x)]\tilde{R}_{\alpha}(y)\dd y.$$
Since $\tilde{R}_{\alpha}\approx |y|^{-Q-2\alpha}$, we have for $\frac{1}{r}=\frac{1}{p_{1}}+\frac{1}{p_{2}}$, $1=\frac{1}{q_{1}}+\frac{1}{q_{2}}$ and $2\alpha=s_{1}+s_{2}$:
\begin{align}
\|\mathcal{H}_{\alpha}(u,v)\|_{L^{r}}&\lesssim \Big(\int_{\G}\Big(\int_{\G}\frac{[u(xy)-u(x)][v(xy)-v(x)]}{|y|^{2\alpha}}\frac{\dd y}{|y|^{Q}}\Big)^{r}\dd x \Big)^{\frac{1}{r}}\notag\\
&\lesssim \int_{\G}\Big(\int_{\G} \Big(\frac{[u(xy)-u(x)][v(xy)-v(x)]}{|y|^{2\alpha}}\Big)^{r}\dd x \Big)^{\frac{1}{r}}\frac{\dd y}{|y|^{Q}}\notag\\
&\lesssim \int_{\G} \frac{1}{|y|^{2\alpha}}\|u(xy)-u(x)\|_{L^{p_{1}}}\|v(xy)-v(x)\|_{L^{p_{2}}}\frac{\dd y}{|y|^{Q}}\notag\\
&\lesssim \|u\|_{\dot{B}_{p_{1},q_{1}}^{s_{1}}}\|v\|_{\dot{B}_{p_{2},q_{2}}^{s_{2}}}\notag
\end{align}
\end{proof}

The case of $L^{p}$ spaces is a little bit more difficult and technical as in \cite{M}. We will see here that for some range of $\alpha$, we can optain a relatively simple proof of some of these bounds and in fact extend the range of the estimates proved in \cite{M} to include a $BMO$ type estimate.

\subsubsection{$L^{p}$-Type Estimates}

\begin{theorem}\label{lp}
Let $\alpha \in (0,\frac{1}{2}]$, then one has
$$\|\mathcal{H}_{\alpha}(u,v)\|_{L^{p}}\lesssim \|(-\Delta_{b})^{\alpha}u\|_{L^{p}}[v]_{BMO}.$$
Moreover, for $\alpha=\alpha_{1}+\alpha_{2}$ with $\alpha_{1},\alpha_{2}\in (0.\frac{1}{2})$ and $\frac{1}{p}=\frac{1}{p_{1}}+\frac{1}{p_{2}}$, one has
$$\|\mathcal{H}_{\alpha}(u,v)\|_{L^{p}}\lesssim \|(-\Delta_{b})^{\alpha_{1}}u\|_{L^{p_{1}}}\|(-\Delta_{b})^{\alpha_{1}}v\|_{L^{p_{2}}}.$$
\end{theorem}

\begin{proof}
We will start first by proving the second claim. We let $h\in L^{p'}$, and we propose to estimate $\int_{\G}\mathcal{H}_{\alpha}(u,v)h\dd x$. Using the fact that 
$$\lim_{t\to 0}\Big(t^{1-2\alpha}\frac{\partial}{\partial t}U_{\alpha}(t,x)\Big)=c_{\alpha}(-\Delta_{b})^{\alpha}u(x)$$
and that
$$\frac{\partial}{\partial t}\Big(t^{1-2\alpha}\frac{\partial}{\partial t}U_{\alpha}\Big)=-t^{1-2\alpha}\Delta_{b}U_{\alpha}$$
we have that
\begin{align}
\left| \int_{\G}\mathcal{H}_{\alpha}(u,v)h\dd x\right|&\approx \left|\int_{\G}\int_{0}^{\infty}\partial_{t}\Big[ t^{1-2\alpha}(U_{\alpha}V_{\alpha}\partial_{t}H_{\alpha}-U_{\alpha}H_{\alpha}\partial_{t}V_{\alpha}-V_{\alpha}H_{\alpha}\partial_{t}U)\Big]\dd t\dd x\right|\notag\\
&=\left|\int_{\G\times \R^{+}}t^{1-2\alpha}\Big[2\partial_{t}U_{\alpha}\partial_{t}V_{\alpha}+\nabla_{\G}V_{\alpha}\nabla_{\G}U_{\alpha}\Big]H_{\alpha}\dd x\dd t\right|\notag\\
&\lesssim \int_{\G\times \R^{+}}t^{2-2\alpha}|\tilde{\nabla}U_{\alpha}||\tilde{\nabla}V_{\alpha}||H_{\alpha}|\frac{\dd x\dd t}{t}.\notag
\end{align}
Now using Theorem \ref{thmest1}, we have that
$$\left|\int_{\G}\mathcal{H}_{\alpha}(u,v)h\dd x\right|\lesssim \|(-\Delta_{b})^{\alpha_{1}}u\|_{L^{p_{1}}}\|(-\Delta_{b})^{\alpha_{2}}v\|_{L^{p_{2}}}\|h\|_{L^{p'}}.$$
Notice that this also provides the proof of the first claim for $\alpha<\frac{1}{2}$ using Theorem \ref{thmest2}. It remains thus to treat the case $\alpha=\frac{1}{2}$. In this case, we have after another integration by parts
\begin{align}
\left|\int_{\G}\mathcal{H}_{\alpha}(u,v)h\dd x\right|&\approx \left|\int_{\G\times \R^{+}}\Big[2\partial_{t}U_{\alpha}\partial_{t}V_{\alpha}+\nabla_{\G}V_{\alpha}\nabla_{\G}U_{\alpha}\Big]H_{\alpha}\dd x\dd t\right|\notag\\
&=\left|\int_{\G\times \R^{+}}t\frac{\partial}{\partial t}\Big[(2\partial_{t}U_{\alpha}\partial_{t}V_{\alpha}+\nabla_{\G}V_{\alpha}\nabla_{\G}U_{\alpha})H_{\alpha}\Big]\dd x\dd t\right|\notag\\
&\lesssim \int_{\G\times \R^{+}}t\Big[|\tilde{\nabla}U_{\alpha}||\tilde{\nabla}V_{\alpha}||\tilde{\nabla}H_{\alpha}|+|\tilde{\nabla}V_{\alpha}||\tilde{\nabla}\nabla_{\G} U_{\alpha}||H_{\alpha}|\Big]\dd x\dd t.\notag
\end{align}
Next writing $t^{2}=tt^{0}t$ for the first term and $t^{2}=tt^{2-2(\frac{1}{2})}$ we get
$$\left|\int_{\G}\mathcal{H}_{\alpha}(u,v)h\dd x\right|\lesssim [v]_{BMO}\|h\|_{L^{p'}}\|(-\Delta_{b})^{\frac{1}{2}}u\|_{L^{p}}.$$
\end{proof}

\subsubsection{Riviere-Da Lio Three Term Commutator}

\begin{theorem}\label{riv}
If $2\alpha\leq 1$ we have
$$\|(-\Delta_{b})^{\alpha}\mathcal{H}_{\alpha}(u,v)\|_{H^{1}}\lesssim \|(-\Delta_{b})^{\alpha}u\|_{L^{p}}\|(-\Delta_{b})^{\alpha}u\|_{L^{p'}}.$$
\end{theorem}

\begin{proof}
Here we will use the duality between $H^{1}$ and BMO (see \cite{FS}). 

\vspace{.5cm}

\textbf{Case $2\alpha<1$:} Let $h\in \mathcal{S}(\G)$ and $g=(-\Delta_{b})^{\alpha}h$. Then we have
\begin{align}
\left|\int_{\G}\mathcal{H}_{\alpha}(u,v)g \dd x\right|&=\left|\int_{\G}uv(-\Delta_{b})^{\alpha}g-ug(-\Delta_{b})^{\alpha}v-vg(-\Delta_{b})^{\alpha}u\dd x\right|\notag\\
&=\left|\int_{\G\times \R^{+}}\partial_{t}(t^{1-2\alpha}(U_{\alpha}V_{\alpha}\partial_{t}G_{\alpha}-U_{\alpha}G_{\alpha}\partial_{t}V_{\alpha}-V_{\alpha}G_{\alpha}\partial_{t}U_{\alpha}))\dd t\dd x\right|\notag\\
&\lesssim \int_{\G\times \R^{+}}t^{1-2\alpha}|\nabla_{\G} U_{\alpha}||\nabla_{\G} V_{\alpha}||G_{\alpha}|\dd t\dd x\notag\\
&\qquad +\left|\int_{\G\times \R^{+}}t^{1-2\alpha}[\partial_{t}U_{\alpha}\partial_{t}V_{\alpha}G_{\alpha}]\dd t\dd x\right|.\notag
\end{align}
We write the first term as $t^{1-2\alpha}|\nabla_{\G} U_{\alpha}|t^{1-2\alpha}|\nabla_{\G} V_{\alpha}|t^{2\alpha}|G_{\alpha}|\frac{1}{t}$ to get, using Theorem \ref{thmest2}, an estimate of the form
$$\int_{\G\times \R^{+}}t^{1-2\alpha}|\nabla_{\G} U_{\alpha}||\nabla_{\G} V_{\alpha}||G_{\alpha}|\dd t\dd x\lesssim [g]_{BMO}\|(-\Delta_{b})^{\alpha}u\|_{L^{p}}\|(-\Delta_{b})^{\alpha}v\|_{L^{q}}.$$
The second term is a little more involved as in the proof of Theorem \ref{lp}, an extra integrations by parts is needed. Indeed,
\begin{align}
\int_{\G\times \R}t^{1-2\alpha}[\partial_{t}U_{\alpha}\partial_{t}V_{\alpha}G_{\alpha}]&=\frac{1}{2\alpha}\int_{\G\times \R^{+}}t^{2\alpha}\partial_{t}(t^{1-2\alpha}\partial_{t}U_{\alpha}t^{1-2\alpha}\partial_{t}V_{\alpha}G_{\alpha})\dd t\dd x\notag\\
&=\frac{-1}{2\alpha}\int_{\G\times\R^{+}}t^{2-2\alpha}(\Delta_{b}U_{\alpha}\partial_{t}V_{\alpha}G_{\alpha}+\Delta_{b}V_{\alpha}U_{\alpha}G_{\alpha}\dd t\dd x\notag\\
&\qquad+\frac{1}{2\alpha}\int_{\G\times \R^{+}}t^{2-2\alpha}\partial_{t}U_{\alpha}\partial_{t}V_{\alpha}\partial_{t}G_{\alpha}\dd t\dd x\notag\\
&=\frac{-1}{2\alpha}\int_{\G\times\R^{+}}t^{2-2\alpha}(\Delta_{b}U_{\alpha}\partial_{t}V_{\alpha}G_{\alpha}+\Delta_{b}V_{\alpha}U_{\alpha}G_{\alpha}\dd t\dd x\notag\\
&\qquad+\frac{1}{8\alpha}\int_{\G\times \R^{+}}t^{4\alpha}\partial_{t}(t^{1-2\alpha}\partial_{t}U_{\alpha}t^{1-2\alpha}\partial_{t}V_{\alpha}t^{1-2\alpha}\partial_{t}G_{\alpha})\dd x\dd t.\notag
\end{align}
Now the first two terms can be easily bounded by the desired quantity. It remains to bound the last one.
\begin{align}
&\int_{\G\times \R^{+}}t^{4\alpha}\partial_{t}(t^{1-2\alpha}\partial_{t}U_{\alpha}t^{1-2\alpha}\partial_{t}V_{\alpha}t^{1-2\alpha}\partial_{t}G_{\alpha})\dd x\dd t\notag\\
&=-\int_{\G\times\R^{+}}t^{3-2\alpha}\Big[
(\Delta_{b}U_{\alpha}\partial_{t}V_{\alpha}+\Delta_{b}V_{\alpha}\partial_{t}U_{\alpha})\partial_{t}G_{\alpha}-(\partial_{t}\nabla_{\G}U_{\alpha}\partial_{t}V_{\alpha}+\partial_{t}\nabla_{\G}V_{\alpha}\partial_{t}U_{\alpha})\nabla_{\G}G_{\alpha}\Big]\dd t\dd x.\notag
\end{align}
Again all the terms here have the right form of Theorems \ref{thmest1} and \ref{thmest2} and they provide the desired bound.

\vspace{.5cm}

\textbf{Case $2\alpha=1$:}
In this case, we have
\begin{align}
|\int_{\G}\mathcal{H}_{\frac{1}{2}}(u,v)g\dd x|&=|\int_{\G\times \R^{+}}\partial_{t}\Big(U_{\frac{1}{2}}V_{\frac{1}{2}}(-\Delta_{b})H_{\frac{1}{2}}-\partial_{t}U_{\frac{1}{2}}V_{\frac{1}{2}}\partial_{t}H_{\frac{1}{2}}-\partial_{t}V_{\frac{1}{2}}U_{\frac{1}{2}}\partial_{t}H_{\frac{1}{2}})\dd t\dd x|\notag\\
&=|\int_{\G\times \R^{+}}t\partial_{t}((U_{\frac{1}{2}}V_{\frac{1}{2}})(-\Delta_{b})H_{\frac{1}{2}}-(U_{\frac{1}{2}}V_{\frac{1}{2}})_{tt}\partial_{t}H_{\frac{1}{2}})\dd t\dd x\notag\\
&=|\int_{\G\times \R^{+}}t\partial_{t}(\tilde{\Delta_{b}}(U_{\frac{1}{2}}V_{\frac{1}{2}})\partial_{t}H_{\frac{1}{2}})\dd t\dd x\notag\\
&=2|\int_{\G\times \R^{+}}t\partial_{t}(\tilde{\nabla}U_{\frac{1}{2}}\tilde{\nabla}V_{\frac{1}{2}}\partial_{t}H_{\frac{1}{2}})\dd t\dd x\notag\\
&\lesssim \int_{\G\times \R^{+}}t|\partial_{t}(\tilde{\nabla}U_{\frac{1}{2}}\tilde{\nabla}V_{\frac{1}{2}})||\partial_{t}H_{\frac{1}{2}})|\dd t\dd x\notag\\
&\qquad+\int_{\G\times \R^{+}}t|\nabla_{\G}(\tilde{\nabla}U_{\frac{1}{2}}\tilde{\nabla}V_{\frac{1}{2}})||\nabla_{\G}H_{\frac{1}{2}}|\dd t\dd x,\notag
\end{align}
where here $\tilde{\Delta_{b}}=\Delta_{b}+\partial_{tt}$ and we used the harmonicity of the extension. Notice here that we can finish the proof as in the previous case.
\end{proof}

Notice that one can also capture the $(L^{p},L^{q})\to L^{r}$ type estimates in \cite{M} by slightly modifying the proof and using the $L^{p}$ estimates for the Riesz potential.

\subsection{Chanillo Type Commutator}
We recall here the commutator estimate proved by Chanillo \cite{Chan}:
$$\|[I_{s},v]u\|_{L^{p}}\lesssim \|u\|_{L^{q}}[v]_{BMO}.$$
Notice that
$$[I_{s},v]u=I_{s}(uv)-vI_{s}(u).$$
Therefore, if we set $u=(-\Delta_{b})^{\frac{s}{2}}a$, we have
$$\int_{\G}(-\Delta_{b})^{\frac{s}{2}}([I_{s},v]u)h=\int_{\G}v[(-\Delta_{b})^{\frac{s}{2}}a-a(-\Delta_{b})^{\frac{s}{2}}h]\dd x.$$
So we propose to estimate an integral of the form
$$\int_{\G}v[(-\Delta_{b})^{\frac{s}{2}}u-u(-\Delta_{b})^{\frac{s}{2}}h]\dd x.$$

\begin{theorem}\label{chan}
For $\frac{1}{p}+\frac{1}{r}-\frac{s}{Q}=1$, we have
$$\left|\int_{\G}v[(-\Delta_{b})^{\frac{s}{2}}u-u(-\Delta_{b})^{\frac{s}{2}}h]\dd x\right|\lesssim [v]_{BMO}\|(-\Delta_{b})^{\frac{s}{2}}u\|_{L^{p}}\|\|(-\Delta_{b})^{\frac{s}{2}}h\|_{L^{r}}.$$

\end{theorem}

\begin{proof}
Again, we use the same trick, that is we write
\begin{align}
\left|\int_{\G}v[h(-\Delta_{b})^{\frac{s}{2}}u-u(-\Delta_{b})^{\frac{s}{2}}h]\dd x\right|&\approx \left|\int_{\G\times \R^{+}}\partial_{t}\Big[t^{1-s}(\partial_{t}U_{\frac{s}{2}}H_{\frac{s}{2}}-\partial_{t}H_{\frac{s}{2}}U_{\frac{s}{2}})V_{\frac{s}{2}}\Big]\dd t\dd x\right|\notag\\
&\lesssim\left|\int_{\G\times \R^{+}}t^{1-s}(\nabla_{\G}U_{\frac{s}{2}}H_{\frac{s}{2}}-\nabla_{\G}U_{\frac{s}{2}}H_{\frac{s}{2}})\nabla_{\G}V_{\frac{s}{2}}\dd t\dd x\notag \right|\\
&\quad+\left|\int_{\G\times \R^{+}}t^{1-s}(\partial_{t}U_{\frac{s}{2}}H_{\frac{s}{2}}-\partial_{t}U_{\frac{s}{2}}H_{\frac{s}{2}})\partial_{t}V_{\frac{s}{2}}\dd t\dd x\right|\notag
\end{align}
Using Theorem \ref{thmest2}, we have that
\begin{align}
\left|\int_{\G\times \R^{+}}t^{1-s}(\nabla_{\G}U_{\frac{s}{2}}H_{\frac{s}{2}}-\nabla_{\G}U_{\frac{s}{2}}H_{\frac{s}{2}})\nabla_{\G}V_{\frac{s}{2}}\dd t\dd x\right|\lesssim [v]_{BMO}\|(-\Delta_{b})^{\frac{s}{2}}u\|_{L^{p}}\|h\|_{L^{q}}\notag\\
\lesssim [v]_{BMO}\|(-\Delta_{b})^{\frac{s}{2}}u\|_{L^{p}}\|(-\Delta_{b})^{\frac{s}{2}}h\|_{L^{r}}\notag
\end{align}
where the second inequality follows from the sobolev embeddings with $\frac{1}{r}=\frac{1}{q}+\frac{s}{Q}$.\\
The second term on the other hand, cannot be bounded directly since we are in the case $s=2\alpha$. That is why, we perform another integration by parts:
\begin{align}
&\int_{\G\times \R^{+}}t^{1-s}(\partial_{t}U_{\frac{s}{2}}H_{\frac{s}{2}}-\partial_{t}U_{\frac{s}{2}}H_{\frac{s}{2}})\partial_{t}V_{\frac{s}{2}}\dd t\dd x\\
&\qquad =\frac{1}{s}\int_{\G\times \R^{+}}t^{s}\partial_{t}(t^{1-s}(\partial_{f}U_{\frac{s}{2}}H_{\frac{s}{2}}-U_{\frac{s}{2}}\partial_{t}H_{\frac{s}{2}})t^{1-s}\partial_{t}V_{\frac{s}{2}})\dd t\dd x\notag\\
&\qquad =\frac{1}{s}\int_{\G\times \R^{+}}-t^{2-s}(\Delta_{b}U_{\frac{s}{2}}H_{\frac{s}{2}}-U_{\frac{s}{2}}\Delta_{b}H_{\frac{s}{2}})\partial_{t}V_{\frac{s}{2}}\notag\\
&\qquad \qquad-t^{2-s}(\partial_{t}U_{\frac{s}{2}}H_{\frac{s}{2}}-U_{\frac{s}{2}}\partial_{t}H_{\frac{s}{2}})\Delta_{b}V_{\frac{s}{2}}\dd t\dd x\notag\\
&\qquad =-\int_{\G\times \R^{+}}t^{2-s}(\Delta_{b}U_{\frac{s}{2}}H_{\frac{s}{2}}-U_{\frac{s}{2}}\Delta_{b}H_{\frac{s}{2}})\partial_{t}V_{\frac{s}{2}}\notag\\
&\qquad \qquad+t^{2-s}(\partial_{t}\nabla_{\G}U_{\frac{s}{2}}H_{\frac{s}{2}}+\partial_{t}U_{\frac{s}{2}}\nabla_{\G}H_{\frac{s}{2}})\nabla_{\G}V_{\frac{s}{2}}\notag\\
&\qquad \qquad-t^{2-s}(\nabla_{\G}U_{\frac{s}{2}}\partial_{t}H_{\frac{s}{2}}-U_{\frac{s}{2}}\partial_{t}\nabla_{\G}H_{\frac{s}{2}})\nabla_{\G}V_{\frac{s}{2}}\dd t\dd x\notag
\end{align}
The first term can be bounded easily as in Theorem \ref{riv}. For the last two terms, we also have the right bound since $s<1+2\alpha=1+s$.
\end{proof}


\begin{thebibliography}{99}

\bibitem {ABCRS1} P. Alonso-Ruiz, F. Baudoin, L. Chen, L. Rogers, N. Shanmugalingam, A. Teplyaev: Besov class via heat semigroup on Dirichlet spaces I: Sobolev type inequalities, Journal of Functional Analysis, Vol. 278 issue 11 (2020) 108459.

\bibitem{ABCRS2} P. Alonso-Ruiz, F. Baudoin, L. Chen, L. Rogers, N. Shanmugalingam, A. Teplyaev:  Besov class via heat semigroup on dirichlet spaces II: Bv functions and gaussian heat kernel estimates. arXiv:1811.11010v2, 2019.

\bibitem{ABCRS3} P. Alonso-Ruiz, F. Baudoin, L. Chen, L. Rogers, N. Shanmugalingam, A. Teplyaev: Besov class via heat semigroup on dirichlet spaces III: Bv functions and sub-gaussian heat kernel estimates. arXiv:1903.10078v1, 2019.

\bibitem{ABB} A. Agrachev, D. Barilari, U. Boscain: Introduction to Riemannian and Sub-Riemannian geometry, available at https://webusers.imj-prg.fr/~davide.barilari/Notes.php

\bibitem{Ba1} H. Bahouri and I. Gallagher: Paraproduit sur le groupe de Heisenberg et applications, Revista Matematica Iberoamericana, 17 (2001), pages 69--105.

\bibitem{Ba} H. Bahouri, C. F. Kammerer, I. Gallagher:  Phase-space analysis and pseudodifferential calculus on the Heisenberg group; Ast\'{e}risque 342, Societe Mathematique de France.

\bibitem{BalT} A. V. Balakrishnan: On the powers of the infinitesimal generators of groups and semigroups of linear bounded
transformations, Thesis (Ph.D.)-University of Southern California. 1954. (no paging), ProQuest LLC

\bibitem{Bal} A. V. Balakrishnan: Fractional powers of closed operators and the semigroups generated by them, Pacific J.
Math. 10 (1960), 419--437.

\bibitem{BG}F. Baudoin, N. Garofalo: A Note on the Boundedness of Riesz Transform for Some Subelliptic Operators; IMNR Volume 2013, Issue 2, 1 January 2013, 398–421.

\bibitem{Bog}K. Bogdan, A. Stos, P. Sztonyk: Harnack inequality for stable processes on d-sets. Studia Mathematica 158(2) (2003), 163-198.

\bibitem{Bon} A. Bonfiglioli, E. Lanconelli, F. Uguzzoni: Stratified Lie groups and potential theory for their sub-Laplacians. Springer Monographs in Mathematics. Springer, Berlin, 2007.

\bibitem{candy} H.Q Bui, T. Candy: A characterization of the Besov-Lipschitz and Triebel-Lizorkin spaces using Poisson like kernels. Functional analysis, harmonic analysis, and image processing: a collection of papers in honor of Bj\"orn Jawerth, 109–141, Contemp. Math., 693, Amer. Math. Soc., Providence, RI, 2017.

\bibitem{Bu1} H.-Q. Bui, M. Paluszy\'nski, M. H. Taibleson: A note on the Besov-Lipschitz and Triebel-Lizorkin
spaces, Contemporary Mathematics 189 (1995), 95--101.

\bibitem{Bu2} H.-Q. Bui, M. Paluszy\'nski, M. H. Taibleson: Characterization of the Besov-Lipschitz and Triebel-Lizorkin spaces, The case $q <1$, The Journal of Fourier Analysis and Applications 3 (1997), 837--846.

\bibitem{CS} L. Caffarelli, L. Silvestre: An extension problem related to the fractional Laplacian. Comm. Partial Differential Equations, 32:1245--1260 ,2007. 

\bibitem{CG} J. Cao, A. Grigor’yan: Heat Kernels and Besov Spaces on Metric Measure Spaces, Preprint.

\bibitem{CDPT07} L. Capogna, D. Danielli, S. Pauls, J. Tyson: An introduction to the Heisenberg group and the sub-Riemannian isoperimetric problem, Birkhauser Progress in Mathematics 259 (2007).

\bibitem{Chan} S. Chanillo: A note on commutators. Indiana Univ. Math. J. 31, no. 1, (1982), 7--16.

\bibitem{Cr}M. Christ: $L^{p}$ Bounds for Spectral Multipliers on Nilpotent Groups. Transactions of the American Mathematical Society, Vol. 328, No. 1 (Nov., 1991), pp. 73--81.

\bibitem{Fran} F. Ferrari, B. Franchi: Harnack inequality for fractional sub-Laplacians in Carnot groups. Math. Z. 279 (1-2), (2015), pp 435--458.

\bibitem{Fer}F. Ferrari, M. Miranda, D. Pallara, A. Pinamonti, Y. Sire: Fractional Laplacians, Perimeters and Heat Semigroups in Carnot Groups. DCDS-S, 2018, 11 (3) : 477--491. doi: 10.3934/dcdss.2018026

\bibitem{DR1} F. Da Lio, T. Rivi\'{e}re: Sub-criticality of non-local Schr\"{o}dinger systems with antisymmetric potentials
and applications to half-harmonic maps. Advances in Mathematics, no. 3,  ( 2011), 1300--1348.

\bibitem{DR2} F. Da Lio and T. Rivi\'{e}re: Three-term commutator estimates and the regularity of 1/2-harmonic maps
into spheres. Analysis and PDE, no. 1, (2011), 149--190.

\bibitem{Fol1} G. Folland: Subelliptic estimates and function spaces on nilpotent Lie groups. Ark. Mat., 13(2), (1975), 161--207.

\bibitem{Fol2} G. Folland: A fundamental solution for a subelliptic operator. Bull. Amer. Math. Soc. 79 (1973), 373--376.

\bibitem{FS} G. Folland, E. Stein: Hardy spaces on homogeneous groups. Mathematical Notes, 28. Princeton University Press, N.J.; University of Tokyo Press, Tokyo, 1982.

\bibitem{FGMT} R.L. Frank, M. del Mar Gonz\'alez, D.D. Monticelli, J. Tan: An extension problem for the CR fractional Laplacian. Adv. Math. 270 (2015), 97--137.

\bibitem{GY} I. Gallagher, Y. Sire: Besov algebras on Lie groups of polynomial growth. Studia Mathematica 212 (2012), 119--139.

\bibitem{Gar1} N. Garofalo; Some properties of Sub-Laplaceans, Electronic Journal of Differential Equations 25 (2018), 103--131.

\bibitem{GG} A. Gover, C.R. Graham: CR invariant powers of the sub-Laplacian, J. Reine Angew. Math. 583 (2005), 1--27.

\bibitem{Gro96} M. Gromov: Carnot-Carath\'{e}odory spaces seen from within, Progress in Mathematics 144 (1996), 79--323.

\bibitem{GJMS} C.R. Graham, R. Jenne, L.J. Mason, G.A.J. Sparling: Conformally invariant powers of the Laplacian I. Existence. J. Lond. Math. Soc. (2) 46(3) (1992) 557--565.

\bibitem{Ken} C. E. Kenig, G. Ponce, L. Vega: Well-posedness and scattering results for the generalized Korteweg-de Vries equation via the contraction principle, Comm. Pure Appl. Math 46, no. 4, (1993), 527--620.

\bibitem{lenz} E. Lenzmann, A. Schikorra: Sharp commutator estimates via harmonic extensions. Nonlinear Anal. 193 (2020), 111375.

\bibitem{M} A. Maalaoui: A Note on Commutators of the Fractional Sub-Laplacian on Carnot Groups. Commun. Pure Appl. Anal. 18 (2019), no. 1, 435--453.

\bibitem{Martini} A. Martini: Spectral theory for commutative algebras of differential operators on Lie groups, J.
Funct. Anal. 260 (2011), no. 9, 2767--2814.

\bibitem{MMC}C. Guidi, A. Maalaoui, V. Martino: Palais-Smale sequences for the fractional CR Yamabe functional and multiplicity results;  Calc. Var. Partial Differential Equations 57 (2018), no. 6, Paper No. 152, 27 pp.

\bibitem{Mon02} R. Montgomery: A tour of subriemannian geometries, their geodesics and applications, American Mathematical Society, Mathematical Surveys and Monographs 91 (2006).

\bibitem{mur}T. Muramatsu: On imbedding theorems for Besov spaces of functions defined in general regions, Publ. Res. Inst. Math. Sci. Kyoto Univ. 7 (1971/1972), 261--285.

\bibitem{PSZ19}A. Pinamonti, G. Speight, S. Zimmerman: A $C^m$ Whitney Extension Theorem for Horizontal Curves in the Heisemberg Group, Transactions of the American Mathematical Society 371(12) (2019), 8971--8992.

\bibitem{PS16} A. Pinamonti, G. Speight: A Measure Zero Universal Differentiability Set in the Heisenberg Group, Mathematische Annalen 368 (1-2), (2017) 233--278.

\bibitem{Ls} L. Roncala, S. Thangavelub: Hardy's inequality for fractional powers of the sublaplacian on the Heisenberg group. Advances in Mathematics, 302, (2016), 106--158.

\bibitem{Saka}  K. Saka: Besov spaces and Sobolev spaces on a nilpotent Lie group, Tohoku Math. Journ. {\bf 31} (1979), 383--437.

\bibitem{Schik} A. Schikorra: $\varepsilon$-regularity for systems involving non-local, antisymmetric operators. Calc. Var. Partial Differential Equations 54, no. 4, (2015), 3531

\bibitem{S} E. Stein: Harmonic Analysis: Real-Variable Methods, Orthogonality, and Oscillatory Integrals, Princeton University Press. 1995.

\bibitem{than} S. Thangavelu: Harmonic Analysis on the Heisenberg Group. Progress in Mathematics 159. Birkhauser, Boston,
MA, 1998.

\bibitem{Yosida}  K. Yosida: Functional Analysis, Springer, Sixth Ed., 1980. 

\bibitem{VSCC} N. Th. Varopoulos, L. Saloff-Coste, T. Coulhon: Analysis and geometry on groups. Cambridge University Press, 1992.

\end{thebibliography}
\end{document}